\documentclass[12pt,draftcls,onecolumn]{IEEEtran}

\IEEEoverridecommandlockouts                              % This command is only needed if you want to  use the \thanks command

\usepackage{amsfonts}
\usepackage{graphicx}
\usepackage{amsmath}
\usepackage{amssymb}
\usepackage{geometry}
\usepackage{fancyhdr}
\usepackage{url}
\usepackage{epsfig}
\usepackage{epstopdf}

\usepackage{subfig}
\usepackage{enumerate}

\newtheorem{theorem}{Theorem}

\newtheorem{definition}{Definition}

\newtheorem{lemma}{Lemma}
\newtheorem{observation}{Observation}

\newtheorem{remark}{Remark}

\title{\LARGE \bf
Lower Bounds on the Performance of Analog to Digital Converters\thanks{}}

\author{Mitra Osqui$\dagger $ \ \ \ \ \ \ \ Alexandre Megretski$\ddagger $ \
\ \ \ \ \ \ Mardavij Roozbehani$\sharp \vspace{-0.15in}$ \thanks{Project partially supported by: Army Research Office ELASTx program.} 
\thanks{$\dagger$Mitra Osqui is currently a Ph.D. candidate at the department of EECS,
Laboratory for Information and Decision Systems (LIDS) at the Massachusetts
Institute of Technology, Cambridge, MA. E-mail: mitra@mit.edu}\thanks{$%
\ddagger $ Alexandre Megretski is currently a professor of EECS at LIDS at
MIT, Cambridge, MA. E-mail: ameg@mit.edu. }\thanks{$\sharp $ Mardavij
Roozbehani is currently a principal research scientist at LIDS at MIT, Cambridge, MA.
E-mail: mardavij@mit.edu. }}

\begin{document}

\maketitle
\thispagestyle{empty}
\pagestyle{empty}

%%%%%%%%%%%%%%%%%%%%%%%%%%%%%%%%%%%%%%%%%%%%%%%%%%%%%%%%%%%%%%%%%%%%%%%%%%%%%%%%
\begin{abstract}

This paper deals with the task of finding certified lower bounds for the performance of Analog to Digital Converters (ADCs). A general ADC is modeled as a causal, discrete-time dynamical system with outputs taking values in a finite set. We define the performance of an ADC as the worst-case average intensity of the filtered input matching error, defined as the difference between the input and output of the ADC. The passband of the shaping filter used to filter the error signal determines the frequency region of interest for minimizing the error. The problem of finding a lower bound for the performance of an ADC is formulated as a dynamic game problem in which the input signal to the ADC plays against the output of the ADC. Furthermore, the performance measure must be optimized in the presence of quantized disturbances (output of the ADC) that can exceed the control variable (input of the ADC) in magnitude. We characterize the optimal solution in terms of a Bellman-type inequality. A numerical approach is presented to compute the value function in parallel with the feedback law for generating the worst case input signal. The specific structure of the problem is used to prove certain properties of the value function that simplifies the iterative computation of a certified solution to the Bellman inequality. The solution provides a certified lower bound on the performance of any ADC with respect to the selected performance criteria. 

\end{abstract}

%%%%%%%%%%%%%%%%%%%%%%%%%%%%%%%%%%%%%%%%%%%%%%%%%%%%%%%%%%%%%%%%%%%%%%%%%%%%%%%%
\section{INTRODUCTION AND MOTIVATION}

Analog to Digital Converters (ADCs) act as the interface between the analog world and digital processors. 
They are present in almost all digital control and communication systems and modern high-speed data 
conversion and storage systems. Naturally, the design and analysis of ADCs have, for many years, attracted 
the attention and interest of researchers from various disciplines across academia and industry. Despite 
the progress that has been made in this field, the design of optimal ADCs remains an open challenging problem, 
and the fundamental limitations of their performance are not well understood. This paper is concerned 
with the latter problem.
\thispagestyle{empty}\pagestyle{empty}

A particular class of ADCs primarily used in high resolution applications
is\ the Delta-Sigma Modulator (DSM). Fig. \ref{fig_t}, illustrates the
classical first-order DSM \cite{PhDBib:Oppenheim}, where $Q$ is a
quantizer with uniform step size.

%%%%%%%%%%%% FIGURE %%%%%%%%%%%%%%%%%
\setlength{\unitlength}{2200sp}
\begingroup\makeatletter\ifx\SetFigFont\undefined\gdef\SetFigFont#1#2#3#4#5{\reset@font\fontsize{#1}{#2pt}\fontfamily{#3}\fontseries{#4}\fontshape{#5}\selectfont}\fi\endgroup
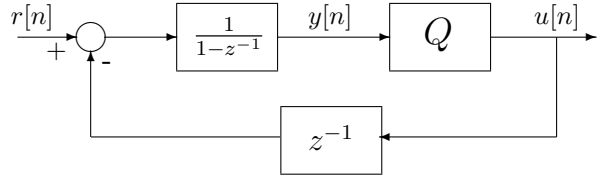
\begin{figure}[h]\begin{center}\begin{picture}(6549,1524)(514,-3298)
\thinlines
\put(1351,-2161){\circle{336}}
\put(2326,-2536){\framebox(1125,750){}}
\put(4726,-2536){\framebox(1125,750){}}

\put(3500,-3700){\framebox(1125,750){}}

\put(1501,-2161){\vector( 1, 0){825}}
\put(3451,-2161){\vector( 1, 0){1275}}
\put(6601,-2161){\line( 0,-1){1125}}

\put(6601,-3286){\vector(-1, 0){2000}}

\put(3485,-3286){\line(-1, 0){2130}}

\put(1351,-3286){\vector( 0, 1){975}}
\put(526,-2161){\vector( 1, 0){675}}
\put(5851,-2161){\vector( 1, 0){1200}}
\put(2500,-2236){\makebox(0,0)[lb]{\smash{\SetFigFont{12}{14.4}{\rmdefault}{\mddefault}{\updefault}$\frac{1}{1-z^{-1}}$}}}

\put(3800,-3436){\makebox(0,0)[lb]{\smash{\SetFigFont{12}{14.4}{\rmdefault}{\mddefault}{\updefault}$z^{-1}$}}}

\put(1470,-2536){\makebox(0,0)[lb]{\smash{\SetFigFont{12}{14.4}{\rmdefault}{\mddefault}{\updefault}-}}}
\put(850,-2400){\makebox(0,0)[lb]{\smash{\SetFigFont{10}{14.4}{\rmdefault}{\mddefault}{\updefault}$+$}}}
%\put(1215,-2250){\makebox(0,0)[lb]{\smash{\SetFigFont{12}{14.4}{\rmdefault}{\mddefault}{\updefault}$+$}}}
\put(450,-2011){\makebox(0,0)[lb]{\smash{\SetFigFont{10}{14.4}{\rmdefault}{\mddefault}{\updefault}$r[n]$}}}
\put(3800,-2011){\makebox(0,0)[lb]{\smash{\SetFigFont{10}{14.4}{\rmdefault}{\mddefault}{\updefault}$y[n]$}}}
\put(5100,-2236){\makebox(0,0)[lb]{\smash{\SetFigFont{14}{14.4}{\rmdefault}{\mddefault}{\updefault}$Q$}}}
\put(6350,-2011){\makebox(0,0)[lb]{\smash{\SetFigFont{10}{14.4}{\rmdefault}{\mddefault}{\updefault}$u[n]$}}}

\end{picture}
\end{center}
\vspace{.05in}
\caption{Classical First-Order Sigma-Delta Modulator}
\vspace{-.05in}
\label{fig_t}
\end{figure}
%%%%%%%%%%%%%%%%%%%%%%%%%%%%%%%%%%%%%%%%%%%

An extensive body of research on DSMs has appeared in the signal processing
literature. One well known approach is based on linearized additive noise
models and filter design for noise shaping \cite{PhDBib:Oppenheim}-\nocite%
{Derpich2008}\nocite{LNM1}\nocite{PhDBib:adhocdesign2}\nocite%
{PhDBib:SchreierTemes2005}\nocite{PhDBib:SchreierTemes2005}\cite%
{PhDBib:Norsworthy1997}. The underlying assumption for validity of the
linearized additive noise model is availability of a relatively high number
of bits. Alternative approaches based on a formalism of the signal
transformation performed\ by the quantizer have been exploited for
deterministic analysis in \cite{PhDBib:ThaoVetterli31}-\nocite%
{PhDBib:ThaoVetterli30}\cite{PhDBib:Thao2006Tile}. Some other works that do
not use linearized additive noise models are reported in \cite{QuevedoGoodwin}-\nocite%
{SteinerYang}\cite{Wang}.

In the control field, \cite{BOYDWang2}-\nocite{BOYDWang}\cite{BOYDWang3} find performance bounds and suboptimal policies for linear stochastic control problems using Bellman inequalities with quadratic value functions. The problem is relaxed and solved using linear matrix inequalities and semidefinite programming. For references on quantized control, please see \cite{Bullo2006}-\nocite{Brockett2000}\cite{EliaMitter}. 
%Work on optimal dynamic quantization can be found in \cite{Azuma2008J}.

In \cite{PhDBib:Mitra1} we provided a characterization of the solution to the optimal ADC design problem and presented a generic methodology for
numerical computation of sub-optimal solutions along with computation of a certified upper bound on the performance. The performance of an ADC is evaluated 
with respect to a cost function which is a measure of the intensity of the error signal (the difference between the input signal and its quantized version) for the worst case input. 
The error signal is passed through a shaping filter which dictates the frequency region in which the error is to be minimized. Furthermore, we showed that the dynamical system within the optimal ADC is a copy of the shaping filter used to define the performance criteria. In \cite{PhDBib:Mitra1} we also presented an exact analytical solution to the optimal ADC for first-order shaping filters, and showed that the classical first-order DSM (Figure \ref{fig_t}) is identical to our optimal ADC. This result proved the optimality of the classical first-order DSM with respect to the adopted performance measure, and was a step towards understanding the limitations of performance.  

In this paper, we present a framework for finding \textit{certified lower bounds} for the performance of ADCs with shaping filters of arbitrary order.
We use the same ADC model and performance measure adopted in \cite{PhDBib:Mitra1}. The objective is to find a lower bound on the infimum of the cost function. 
The approach is to find a feedback law for generating the input of the ADC such that regardless of its output, the performance is bounded from below 
by a certain value. Thus, the input of the ADC is viewed as the control, and the problem is posed within a non-linear optimal feedback control/game framework. We show that the optimal control law can be characterized in terms of a \textit{value function} satisfying an analog of the 
Bellman inequality. The value function in the Bellman inequality and the corresponding control law can be jointly computed via 
value iteration. 

%We emphasize that the value function corresponding to the lower bound problem is very different than the value function associated with the upper bound problem previously reported in \cite{PhDBib:Mitra1}.

Since searching for the value function involves solving a sequence of infinite dimensional optimization problems, some approximations are needed for numerical computation. First, a finite-dimensional parameterization of the value function is selected. Second, the state space and 
the input space are discretized. Third, the computations are restricted to a finite subset of the space. The latter step deserves 
further elaboration. If the dynamical system inside the ADC is strictly stable, then a bounded control invariant set exists, thus it is possible 
to do the computations over a bounded region. The challenge arises when the filter has poles on the unit circle. In this case, 
there does not exist a bounded control invariant set, since the disturbances can exceed the control variable in magnitude. Under the condition that there is at most one pole on the unit circle, we present a theorem that states that the value function is zero outside a certain bounded space. Thus, we have an a priori knowledge of an analytic expression for the value function beyond a bounded region. As a result, 
the computations need to be carried out only over this bounded region. This is in dramatic contrast with the case of upper bound computations \cite{PhDBib:Mitra1}, something to be discussed in section \ref{OurApp}.

%The main contributions of this paper are as follows.

The organization is as follows. Section \ref{ProbForm}
provides a rigorous problem formulation. The main contributions are
presented in Section \ref{OurApp} and \ref{TexSta}. Section \ref{OurApp} describes
our methodology for finding certified lower bounds for ADCs. Section \ref{TexSta} provides
our theoretical results. We provide an example in section \ref{NumEx}, and section \ref{Future} concludes the paper.
%use our algorithm to provide a certified lower bound for the performance of the ADC from the example in \cite{PhDBib:Mitra1}, for which we had given a suboptimal design for the ADC with respect to the selected shaping filter along with an upper bound on its performance. Finally, Section \ref{Future} concludes the paper.

\subsection*{Notation and Terminology:}
\begin{itemize}

\item Function $f:\mathbb{R}^m \mapsto \mathbb{R}$ is called BIBO, if the image of every bounded subset $\Omega\subset\mathbb{R}^{m}$ under $f$, $f(\Omega)$, is bounded.
%Function $f:\mathbb{R}^m \mapsto \mathbb{R}$ is called BIBO if condition 
%\begin{equation}
%\sup_{x\in\Omega}|f(x)|<\infty
%\end{equation}
%holds for every bounded set $\Omega\subset\mathbb{R}^{m}$.
\medskip

\item Given a set $P$, $\ell_+(P)$ is the set of all one-sided sequences $x$ with values in $P$, i.e. functions $x:\mathbb{Z}_+ \mapsto P$.
%Given a set $P$, $\ell_+(P)$ is the set of all sequences that map $\mathbb{Z}_+$ to $P$:
%\begin{align}
%\ell_+(P) & \;  {\buildrel\rm def\over =}\; \left\{x:\mathbb{Z}_+ \mapsto P \right\}. 
%\end{align}

\medskip

\item The $\infty-$norm is defined as:
\begin{equation*}
\|v\|_\infty = \max |v_i|, \quad \text{for} \quad v =  \left[
\begin{array}
[c]{ccc}
v_1 \\  \vdots \\  v_m
\end{array}
\right] \in  \mathbb{R}^m
\end{equation*}
and
\begin{equation*}
\|M\|_\infty  \;  {\buildrel\rm def\over =}\; \sup_{v \ne 0}\frac{\|Mv\|_\infty}{\|v\|_\infty} =\max_{i \in\{1, \cdots, l\}} \sum_{j=1}^{m}|M_{ij}|
\end{equation*}
for a matrix $M = \left(M_{ij}\right) \in \mathbb{R}^{l \times m}$.

%The $\infty-$norm of a vector $v = (v_1, \cdots, v_m)$ in $\mathbb{R}^m$ is defined as:
%\begin{equation}
%\|v\|_\infty = \max_{i \in\{1, \cdots, m\}} |v_i|
%\end{equation}
%The $\infty-$norm of a matrix $M \in \mathbb{R}^{ l \times m}$ is defined as:
%\begin{align}
%\|M\|_\infty  & \;  {\buildrel\rm def\over =}\; \sup_{v \ne 0}\frac{\|Mv\|_\infty}{\|v\|_\infty}\\
%                   & =\max_{i \in\{1, \cdots, l\}} \sum_{j=1}^{m}|\mu_{ij}|
%\end{align}
%where $\mu_{ij}$ is the $(i,j)^{th}$ element of $M$.

\item Let $X$ be a set and $f:X \mapsto \mathbb{R}$ be a function. For every $\epsilon > 0$,
%\begin{equation}
%\arg^\epsilon \inf_{x \in X} f(x)  \;  {\buildrel\rm def\over =}\; \left \{ x \in X : f(x) < \epsilon +  \inf_{x \in X} f(x) \right \}
%\end{equation}
%and
\begin{equation}
\arg^\epsilon \sup_{x \in X} f(x)  \;  {\buildrel\rm def\over =}\; \left \{ x \in X : f(x) > -\epsilon +  \sup_{x \in X} f(x) \right \}.
\end{equation}

\end{itemize}
%\begin{table}
%\caption{An Example of a Table}
%\label{table_example}
%\begin{center}
%\begin{tabular}{|c||c|}
%\hline
%One & Two\\
%\hline
%Three & Four\\
%\hline
%\end{tabular}
%\end{center}
%\end{table}

%%%%%%%%%%%%%%%%%%%%%%%%%%%%%%%%%%%%%%%%%%%%%%%%%%%%%%%%%%%%%%%%%%%%%%%%%%%%%%%%
\section{PROBLEM FORMULATION\label{ProbForm}}

The problem setup in this section is taken from \cite{PhDBib:Mitra1}.

%%%%%%%%%%%%
\subsection{Analog to Digital Converters}

In this paper, a general ADC is viewed as a causal, discrete-time, non-linear
system $\Psi,$ accepting arbitrary inputs in the $[-1,1]$ range, and producing
outputs in a fixed finite subset $U\subset\mathbb{R},$ as shown in Fig.
\ref{fig_a}. We assume that the smallest element in the set $U$ is less than $-1$ and the largest element is greater than $1$.\vspace{0.15in}%

\setlength{\unitlength}{1800sp}
\begingroup\makeatletter\ifx\SetFigFont\undefined\gdef\SetFigFont
#1#2#3#4#5{\reset@font\fontsize{#1}{#2pt}\fontfamily{#3}\fontseries
{#4}\fontshape{#5}\selectfont}\fi\endgroup\begin{figure}[h]\begin{center}%
\begin{picture}(6549,1524)(514,-3298)
\thinlines\put(2350,-2736){\framebox(2200,1250){}}
\put(800,-2161){\vector( 1, 0){1570}}
\put(4550,-2161){\vector( 1, 0){1275}}
\put(3200,-2300){\makebox(0,0)[lb]{\smash{\SetFigFont{18}{14.4}{\rmdefault}{\mddefault}{\updefault}$\Psi$}}}
\put(200,-1900){\makebox(0,0)[lb]{\smash{\SetFigFont{10}{14.4}{\rmdefault}{\mddefault}{\updefault}$r[n] \in[-1,1]$}}}
\put(200,-2560){\makebox(0,0)[lb]{\smash{\SetFigFont{10}{14.4}{\rmdefault}{\mddefault}{\updefault}$n \in\mathbb{Z}_+$}}}
\put(5000,-1900){\makebox(0,0)[lb]{\smash{\SetFigFont{10}{14.4}{\rmdefault}{\mddefault}{\updefault}$u[n]\in U$}}}
\put(5000,-2560){\makebox(0,0)[lb]{\smash{\SetFigFont{10}{14.4}{\rmdefault}{\mddefault}{\updefault}$n \in\mathbb{Z}_{+}$}}}
\end{picture}
\end{center}
\vspace*{-.2in}
\caption{Analog to Digital Converter as a Dynamical System}
\vspace{-.05in}
\label{fig_a}
\end{figure}
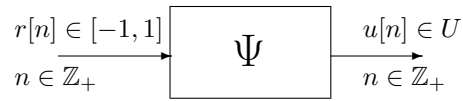%

Equivalently, an ADC is defined by a sequence of functions
$\Upsilon_{n}:\left[  -1,1\right]  ^{n+1}\mapsto U$ according to
\smallskip
\begin{equation}
\Psi:~u[n]=\Upsilon_{n}\left(  r[n]  ,r[ n-1], \cdots ,r[0]  \right)
,~~n\in\mathbb{Z}_{{\small +}}.\label{general_u}%
\end{equation}

The class of ADCs defined above is denoted by $\mathcal{Y}_{U}.$

%%%%%%%%%%%%%%%%%%%%%%
\subsection{Asymptotic Weighted Average Intensity (AWAI) of a Signal}

The Asymptotic Weighted Average Intensity $\eta_{G,\phi}\left(  w\right)$ of a signal $w$ with respect to the transfer function
$G\left(  z\right)  $ of a strictly causal LTI\ dynamical system $L_{G}$ and a non-negative function $\phi:\mathbb{R}\mapsto\mathbb{R}_{+}$ is given by:
\begin{equation}
\eta_{G,\phi}\left(  w\right)  =\underset{N\mapsto\infty}{\lim\sup}%
\frac{1}{N+1}%
{\displaystyle\sum\limits_{n=0}^{N}}
\phi\left(  q[n]  \right)  ,\label{AWA}%
\end{equation}
\noindent where the sequence $q$ is the response to input $w$ of the dynamical system: 
%\begin{align}
%x[n+1]   & =Ax[n]  +Bw[n], \quad x[0] =0, \quad\forall n\in\mathbb{Z}_{+}\label{LG0}\\
%q[n]   & =Cx[n],\label{LG2}%
%\end{align}

\begin{equation}
L_G:
\begin{array}
[c]{cc}
x[n+1]  =Ax[n]  +Bw[n], \quad x[0] =0, \label{LG0}\\ q[n]  =Cx[n]
\end{array}
\end{equation}

\noindent and $A,$ $B,$ $C$ are given matrices of appropriate dimensions. Examples of functions $\phi$ to consider are: $\phi(q )=\left\vert q \right \vert$ and $\phi(q)=\left\vert q \right \vert^2$.

%
%\setlength{\unitlength}{2200sp}\begingroup\makeatletter\ifx\SetFigFont
%\undefined\gdef\SetFigFont#1#2#3#4#5{  \reset@font\fontsize{#1}{#2pt}
%\fontfamily{#3}\fontseries{#4}\fontshape{#5}  \selectfont}\fi\endgroup
%\begin{figure}[h]\begin{center}\begin{picture}(8649,1599)(289,-2773)
%\thinlines\put(2576,-2086){\framebox(1350,900){}}
%\put(1550,-1636){\vector( 1, 0){1000}}
%\put(3926,-1636){\vector( 1, 0){1000}}
%\put(1600,-1486){\makebox(0,0)[lb]{\smash{\SetFigFont{10}{16.8}{\rmdefault}{\mddefault}{\updefault}$w[n]$}}}
%\put(4250,-1486){\makebox(0,0)[lb]{\smash{\SetFigFont{10}{16.8}{\rmdefault}{\mddefault}{\updefault}$q[n]$}}}
%\put(3000,-1700){\makebox(0,0)[lb]{\smash{\SetFigFont{14}{16.8}{\rmdefault}{\mddefault}{\updefault}$L_G$}}}
%\end{picture}
%\end{center}
%\vspace*{-.3in}
%\caption{Strictly Proper LTI Shaping Filter $G(z)$}
%\vspace*{-.2in}
%\label{fig_b}
%\end{figure}
%%%%%%%%%%%%%%
%\smallskip
\subsection{ADC Performance Measure\label{PerfMeas}}

The setup that we use to measure the performance of an ADC is illustrated in
Fig. \ref{fig1}. The performance measure of $\Psi\in\mathcal{Y}_{U}$,
denoted by $\mathcal{J}_{G,\phi}\left(  \Psi\right)  ,$ is the worst-case AWAI
of the error signal for all input sequences $r \in \ell_+(\left[
-1,1\right])  ,$ that is:%
\begin{equation}
\mathcal{J}_{G,\phi}\left(  \Psi\right)  =\sup_{r\in\ell_+(\left[-1,1\right])}\eta_{G,\phi
}\left( r - \Psi\left(  r\right)\right). \label{perfmeas}
\end{equation}

\setlength{\unitlength}{2200sp}\begingroup\makeatletter\ifx\SetFigFont
\undefined\gdef\SetFigFont#1#2#3#4#5{  \reset@font\fontsize{#1}{#2pt}
\fontfamily{#3}\fontseries{#4}\fontshape{#5}  \selectfont}\fi\endgroup
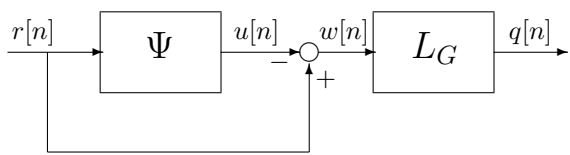
\begin{figure}[h]\begin{center}\begin{picture}(8649,1599)(289,-2773)
\thinlines\put(3850,-1636){\circle{212}}
\put(1500,-2086){\framebox(1350,900){}}
\put(4576,-2086){\framebox(1350,900){}}
\put(2850,-1636){\vector( 1, 0){900}}
\put(3950,-1636){\vector( 1, 0){650}}
\put(5926,-1636){\vector( 1, 0){840}}
\put(3850,-2761){\vector( 0, 1){1000}}
\put(900,-2761){\line(1, 0){2950}}
\put(900,-2761){\line( 0, 1){1125}}
\put(450,-1636){\vector( 1, 0){1080}}
\put(2000,-1700){\makebox(0,0)[lb]{\smash{\SetFigFont{14}{16.8}{\rmdefault}{\mddefault}{\updefault}$\Psi$}}}
\put(3920,-2011){\makebox(0,0)[lb]{\smash{\SetFigFont{10}{16.8}{\rmdefault}{\mddefault}{\updefault}$+$}}}
%\put(3750,-1711){\makebox(0,0)[lb]{\smash{\SetFigFont{10}{16.8}{\rmdefault}{\mddefault}{\updefault}$+$}}}
\put(3400,-1820){\makebox(0,0)[lb]{\smash{\SetFigFont{10}{16.8}{\rmdefault}{\mddefault}{\updefault}$-$}}}
\put(500,-1486){\makebox(0,0)[lb]{\smash{\SetFigFont{10}{16.8}{\rmdefault}{\mddefault}{\updefault}$r[n]$}}}
\put(3000,-1486){\makebox(0,0)[lb]{\smash{\SetFigFont{10}{16.8}{\rmdefault}{\mddefault}{\updefault}$u[n]$}}}
\put(3960,-1486){\makebox(0,0)[lb]{\smash{\SetFigFont{10}{16.8}{\rmdefault}{\mddefault}{\updefault}$w[n]$}}}
\put(6100,-1486){\makebox(0,0)[lb]{\smash{\SetFigFont{10}{16.8}{\rmdefault}{\mddefault}{\updefault}$q[n]$}}}
\put(5000,-1700){\makebox(0,0)[lb]{\smash{\SetFigFont{14}{16.8}{\rmdefault}{\mddefault}{\updefault}$L_G$}}}
\end{picture}
\end{center}
\vspace*{0.0in}
\caption{Setup Used for Measuring the Performance of the ADC}
\vspace*{-.1in}
\label{fig1}
\end{figure}

%that when $\phi\left(  \cdot\right)  =\left\vert
%\cdot\right\vert ^{2}$ and $L_{G}$ is a strictly stable dynamical system with
%transfer function $G\left(  z\right)  ,$ the AWA can be interpreted as the
%average power of the filtered input for signals which are sums of sinusoids:%
%\begin{equation}
%w\left[  n\right]  =%
%%TCIMACRO{\dsum \limits_{k=0}^{\infty}}%
%%BeginExpansion
%{\displaystyle\sum\limits_{k=0}^{\infty}}
%%EndExpansion
%w_{k}e^{j\omega_{k}n}\mapsto\eta_{G,\phi}\left(  w\right)  =%
%%TCIMACRO{\dsum \limits_{k=0}^{\infty}}%
%%BeginExpansion
%{\displaystyle\sum\limits_{k=0}^{\infty}}
%%EndExpansion
%\left\vert w_{k}\right\vert ^{2}\left\vert G\left(  e^{j\omega_{k}}\right)
%\right\vert ^{2}.\label{sumsin}%
%\end{equation}
%Therefore, the AWA allows for penalizing the input of the filter over the
%frequency range of interest via the pass-band of $G\left(  z\right)  .$ An
%alternative measure can be obtained with $\phi\left(  \cdot\right)
%=\left\vert \cdot\right\vert ,$ which is attractive due to its simplifying
%properties. In this case, the AWA represents the average amplitude of the
%filtered input signal.

%%%%%%%%%%%%%%%%
\subsection{ADC Optimization}

Given $L_{G}$ and $\phi,$ we consider $\Psi_{o}\in\mathcal{Y}_{U}$ an optimal
ADC if $\mathcal{J}_{G,\phi}\left(  \Psi_{o}\right)  \leq\mathcal{J}_{G,\phi
}\left(  \Psi\right)  $ for all $\Psi\in\mathcal{Y}_{U}.$ The corresponding
optimal performance measure $\gamma_{G,\phi}\left(  U\right)  $ is defined as%
\begin{equation}
\gamma_{G,\phi}\left(  U\right)  =\underset{\Psi\in\mathcal{Y}_{U}}{\inf
}\mathcal{J}_{G,\phi}\left(  \Psi\right).\label{cutefuzzfuzz}%
\end{equation}

The objective is to find certified lower bounds for (\ref{cutefuzzfuzz}).

%%%%%%%%%%%%%%%%%%%%%%%%%%%%%%%%%%%%%%%%%%%%%%%%%%%%%%%%%%%%%%%%%%%%%%%%%%%%%%%%
\section{OUR APPROACH\label{OurApp}}
We find the lower bound on the performance of any given ADC belonging to the class $\mathcal{Y}_{U}$ by
associating the problem with a full-information feedback control problem. The objective is to 
find a feedback law for generating the input of the ADC, $r$, such that regardless of the output $u$, the performance is bounded from below
by a certain value. Thus, in this setup, $r$ is viewed as the control and $u$ is viewed as the input of a strictly causal system with output $r$.  The setup is depicted in Fig. \ref{figStateFeedback}, where the function $K_r:\mathbb{R}^{m} \mapsto \left[  -1,1\right]$ is said to be an admissible controller if there exists $\gamma\in\lbrack0,\infty)$ such that every triplet of sequences $(x,u,r)$ satisfying
\begin{align}
x[n+1]   & =Ax[n]  +Br[n] - Bu[n], \quad x[0] =0,\label{x+}\\
r[n]   & =K_r\left(  x[n]\right)  ,\label{OptUK}\\
q[n]   & =Cx[n]  ,\label{qPsi}%
\end{align}
\noindent also satisfies the dissipation inequality
\begin{equation}
\inf_{N}{\displaystyle\sum\limits_{n=0}^{N}}\left(\phi\left(  q[n]  \right) - \gamma\right)
>-\infty. \label{optUcondition}
\end{equation}

Note that if \eqref{optUcondition} holds subject to \eqref{x+}-\eqref{qPsi}, then $\gamma_{G,\phi}\left(  U\right)
\geq\gamma.$ Let $\gamma_{o}$ be the minimal upper bound of $\gamma$, for
which an admissible controller exists. Then $K_r$ is said to be an optimal
controller if (\ref{optUcondition}) is satisfied with $\gamma=\gamma_{o}.$  \vspace{-.1in}
%%%%%%%%%%%%%%%%%%
%It is important to note that in \cite{PhDBib:Mitra1}, since the objective was to design the ADC and to find an upper bound on its performance, the control was the output, $u$, of the ADC. We assumed that the largest element in the set $U$ is strictly greater than the largest value that the input $r$ could take. This assumption is necessary to ensure that, in the case when $\max|eig(A)|=1$, there exists an ADC such that the performance measure is finite.
%%%%%%%%%%%%%%%%This assumption was necessary to ensure that in the case when $max|eig(A)|=1$ there would exist a control such that the state trajectory was bounded. 
%Therefore, this constraint on the set $U$ is present for finding the lower bound as well. However, this presents a difficulty for finding the lower bound, since the roles of $r$ and $u$ are reversed with $r$ being the control, which takes values strictly smaller than the largest value that $u$ can take. Therefore, if $A$ has eigenvalues with unit magnitude, a bounded control invariant set will not exist. We will return to this issue shortly in subsection \ref{NumSolnBell}.

%%%%%%%%%%%%%%%%%%%%%%%%
\setlength{\unitlength}{2200sp}\begingroup\makeatletter\ifx\SetFigFont
\undefined\gdef\SetFigFont#1#2#3#4#5{  \reset@font\fontsize{#1}{#2pt}
\fontfamily{#3}\fontseries{#4}\fontshape{#5}  \selectfont}\fi\endgroup
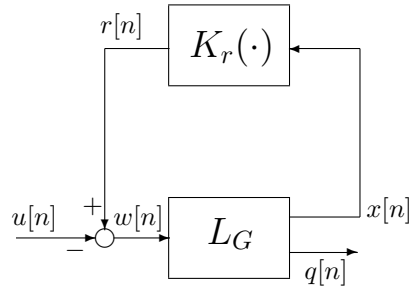
\begin{figure}[h]\begin{center}\begin{picture}(10649,3200)(1289,-1900)
\centering
%\thicklines\put(800,-3436){\dashbox{120}(6250,2950){}}
%\thinlines\put(6470,-3100){\makebox(0,0)[lb]{\smash{\SetFigFont{10}{14.4}{\rmdefault}{\mddefault}{\updefault}$u[n]$}}}
%\put(3201,-2850){\oval(150,150)[br]}
%\put(3201,-2860){\oval(150,150)[tr]}
%\put(3201,-1640){\line( 0,-1){1130}}
%\put(3201,-2920){\line( 0,-1){360}}
%\put(3201,-3270){\vector( 1, 0){4100}}
%\put(4000,-300){\makebox(0,0)[lb]{\smash{\SetFigFont{14}{14.4}{\rmdefault}{\mddefault}{\updefault}$\Psi$}}}
%\put(930,-1000){\makebox(0,0)[lb]{\smash{\SetFigFont{10}{14.4}{\rmdefault}{\mddefault}{\updefault}$x_{\Psi}[n]$}}}
%\put(870,-1300){\vector( 1, 0){610}}
%\put(870,-1300){\line( 0, 1){600}}
%\put(870, -700){\line( 1, 0){5370}}
\put(3850,500){\line( 1, 0){730}}
\put(3850, 500){\vector( 0,-1){2050}}
\put(3800,670){\makebox(0,0)[lb]{\smash{\SetFigFont{10}{16.8}{\rmdefault}{\mddefault}{\updefault}$r[n]$}}}
\put(6726,500){\vector( -1, 0){800}}
\put(6726, 500){\line( 0,-1){1900}}
\put(5926,-1400){\line( 1, 0){800}}
\put(6800,-1400){\makebox(0,0)[lb]{\smash{\SetFigFont{10}{16.8}{\rmdefault}{\mddefault}{\updefault}$x[n]$}}}
\put(3850,-1636){\circle{212}}
%\put(1500,-2086){\framebox(1350,900){}}
\put(4576,86){\framebox(1350,900){}}
\put(4576,-2086){\framebox(1350,900){}}
\put(2850,-1636){\vector( 1, 0){900}}
\put(3950,-1636){\vector( 1, 0){635}}
\put(5926,-1800){\vector( 1, 0){775}}
%\put(3850,-2855){\vector( 0, 1){1125}}
%\put(900,-2855){\line(1, 0){2950}}
%\put(900,-2855){\line( 0, 1){1125}}
%\put(440,-1736){\vector( 1, 0){1050}}
%\put(1800,-1700){\makebox(0,0)[lb]{\smash{\SetFigFont{14}{16.8}{\rmdefault}{\mddefault}{\updefault}$K(\cdot)$}}}
%\put(3750,-1711){\makebox(0,0)[lb]{\smash{\SetFigFont{10}{16.8}{\rmdefault}{\mddefault}{\updefault}$+$}}}
\put(3400,-1850){\makebox(0,0)[lb]{\smash{\SetFigFont{10}{16.8}{\rmdefault}{\mddefault}{\updefault}$-$}}}
\put(3600,-1350){\makebox(0,0)[lb]{\smash{\SetFigFont{10}{16.8}{\rmdefault}{\mddefault}{\updefault}$+$}}}

\put(2800,-1486){\makebox(0,0)[lb]{\smash{\SetFigFont{10}{16.8}{\rmdefault}{\mddefault}{\updefault}$u[n]$}}}
\put(3950,-1486){\makebox(0,0)[lb]{\smash{\SetFigFont{10}{16.8}{\rmdefault}{\mddefault}{\updefault}$w[n]$}}}
\put(6100,-2100){\makebox(0,0)[lb]{\smash{\SetFigFont{10}{16.8}{\rmdefault}{\mddefault}{\updefault}$q[n]$}}}
\put(5000,-1700){\makebox(0,0)[lb]{\smash{\SetFigFont{14}{16.8}{\rmdefault}{\mddefault}{\updefault}$L_G$}}}
\put(4800,400){\makebox(0,0)[lb]{\smash{\SetFigFont{14}{16.8}{\rmdefault}{\mddefault}{\updefault}$K_r(\cdot)$}}}

\end{picture}
\end{center}
\vspace*{0.2in}
\caption{Full State-Feedback Control Setup}
\vspace*{-.15in}
\label{figStateFeedback}
\end{figure}

%%%%%%%%%%%%%%%%
\subsection{The Bellman Inequality}

The solution to a well-posed state-feedback optimal control problem can be
characterized as the solution to the associated Bellman equation
\cite{PhDBib:Helton}-\nocite{PhDBib:Bertsekas1}\nocite{PhDBib:Bertsekas2}\cite{PhDBib:Bellman}. Herein, standard techniques are used to show that a
controller $K_r$ satisfying (\ref{optUcondition}) exists if and only
if a solution to an analog of the Bellman equation exists. The formulation
will be made more precise as follows. Define function $\sigma_{\gamma}:\mathbb{R}
^{m}\mathbb{\mapsto R}$ by
\begin{equation}
\sigma_{\gamma}\left(  x\right)  =\gamma - \phi\left(  Cx\right).
\label{gam_sig}
\end{equation}

It can be shown that a controller $K_r$ in \eqref{OptUK} guaranteeing \eqref{optUcondition} exists if and only if there exists a function $V:\mathbb{R}^{m}\mapsto\mathbb{R}_{+}$, such that inequality
\begin{equation}
V\left(  x\right)  \geq\sigma_{\gamma}\left(  x\right)  +\inf_{r\in\left[
-1,1\right]  }\max_{u\in U}V\left(  Ax+Br-Bu\right) \label{bellman}
\end{equation}
holds for all $x\in\mathbb{R}^{m}$ (see Theorem \ref{dpThmLB}). We refer to
inequality \eqref{bellman} as the Bellman inequality, and to a function $V$ satisfying \eqref{bellman} as the value function.

%%%%%%%%%%%%%%
\subsection{Numerical Solutions to the Bellman Inequality\label{NumSolnBell}}

In this section, we outline our approach for numerical computation of the value function $V$ and the control function $K_r$. 
We can simplify the problem of searching for a solution to inequality (\ref{bellman}) by instead finding a solution $V\ge0$ to the inequality
 \begin{equation}
V\left(  x\right)  \geq\sigma_{\gamma}\left(  x\right)  +\min_{r\in \Gamma_r  }\max_{u\in U}V\left(  Ax+Br-Bu\right),~\forall x\in\mathbb{R}^{m}\label{bellman2}%
\end{equation}

\noindent where $\Gamma_r$ is a finite subset of $\left[  -1,1\right]$. Since for every function $g:\left[  -1,1\right]\mapsto \mathbb{R}$, we have
 \begin{equation}
\inf_{r \in \left[  -1,1\right] } g\left(r\right) \leq \min_{r\in \Gamma_r}g\left(r\right),\label{infmin}%
\end{equation}

\noindent a solution $V$  of (\ref{bellman2}) is  also a solution of  (\ref{bellman}). In the remainder of this section we focus on finding a 
solution to (\ref{bellman2}).

A control invariant set of system (\ref{x+}), with respect to $\Gamma_r$, is formally defined as a subset
$\mathcal{I}\subset\mathbb{R}^{m}$ such that:
\begin{equation}
\forall x\in\mathcal{I},\text{ }\exists r\in\Gamma_r:~Ax+Br-Bu\in\mathcal{I},\text{~}\forall u\in U  .\label{CInvSet}
\end{equation}

Furthermore, a strong invariant set of system (\ref{x+}), with respect to $\Gamma_r$, is defined as a subset
$\mathcal{I}\subset\mathbb{R}^{m}$ such that:
\begin{equation}
Ax+Br-Bu\in\mathcal{I},\quad \forall x\in\mathcal{I},~ r\in\Gamma_r,~ u\in U  .\label{StrongCInvSet}
\end{equation}

%The case where $\max|eig(A)|<1$ is trivial. Since the reachable set $Q$ is bounded, every subset of $Q$ is also bounded. Thus, the set $M$ in \eqref{Vne0Bounded} is bounded.

Ideally we would like to have a bounded invariant set, so that the search for $V$ satisfying the Bellman inequality is restricted to a bounded region of the state space. If $\max|\operatorname{eig}(A)|<1$, then a bounded set $\mathcal{I}$ satisfying (\ref{StrongCInvSet}) is guaranteed to exist. However, if $\max|\operatorname{eig}(A)|=1$, then there is no bounded set $\mathcal{I}$ satisfying (\ref{CInvSet}), due to the assumption that the smallest element in the set $U$ is less than $-1$ and the largest element is greater than $1$. Hence, the case when $\max|\operatorname{eig}(A)|=1$ presents the challenge of searching for a numerical solution to (\ref{bellman2}) over an unbounded state space.  However, for the case that there is only one pole on the unit circle, we will establish in Theorem \ref{boundednessTHM} that the value function is zero for all $x$ outside a certain bounded region. Hence, the numerical search for $V$  satisfying $\left(  \ref{bellman2}\right)$ needs to be carried out only over a bounded subset of the state space.
%%%%%%%Note that we did not encounter this problem when searching for solutions to the corresponding Bellman inequality for design of the ADC and finding the upper bound in \cite{PhDBib:Mitra1}. Because, for the upper bound, the control was the output of the ADC $u$, which takes values strictly larger than $1$, thus even if $A$ has eigenvalues with magnitude equal to $1$, there will exist control that can overcome the input and keep the state space invariant. For the lower bound, the roles of $r$ and $u$ are reversed with $r$ being the control, which is strictly smaller than the largest value that $u$ can take, hence if $A$ has eigenvalues with unit magnitude, a control invariant set will not exist.
Next, uniform grids are created for the state space. In this paper, these are uniformly-spaced, discrete
subsets of the Euclidean space, and are defined as follows. The set 
\begin{equation*}
\mathbb{G}=\left\{  i\Delta~|~i\in\mathbb{Z}\right\}
\end{equation*}
is a grid on $\mathbb{R}$, where $D=1/\Delta$ is a positive integer. The
corresponding grid on $\mathcal{I}$ is 
\begin{equation*}
\Gamma  =\mathbb{G}^{m}\cap\mathcal{I}.
\end{equation*}

Furthermore, we define $ \Gamma_{r}=\{r_1,~r_2,\cdots,r_L\}$ as
\[
\Gamma_{r}   =\mathbb{G}\cap\left[  -1,1\right].
\]

The next step is to create a finite-dimensional parameterization of $V.$ In
this paper, the search is performed over the class of \textit{piecewise constant}
(PWC) functions assuming a constant value over a \textit{tile}. A
\textit{tile} in $\mathbb{G}^{n},$ $n\in\mathbb{N}$ is defined as the smallest
hypercube formed by $2^{n}$ points on the grid, and thus, has $2n$ faces (the
faces are hypercubes of dimension $n-1$). By convention, we assume that the
$n$ faces that contain the lexicographically smallest vertex are closed, and
the rest are open. The union of all such tiles covers $\mathbb{R}^{n}$ and
their intersection is empty. Let $T_{i}$ denote the $i^{th}$ tile over
the grid $\mathbb{G}^{m}$, and $\mathcal{T}$ the set of all tiles that %%%%%%fully
lie within $\mathcal{I},$ and $N_{T}$ the number of all such tiles:
\[
\mathcal{T=}\left\{  T_{i}~|~i\in\left\{  1,2,\cdots,N_{T}\right\}  \right\}.
\]
The PWC parameterization of $V$ is as follows
\begin{equation}
V\left(  x\right)  =V_{i},~\forall x\in T_{i},~i\in\left\{  1,2,\cdots
,N_{T}\right\} \label{PWMLV}
\end{equation}
where $V_{i}\in\mathbb{R}_{+}$. We then search for a solution $V:\mathcal{I}\mapsto \mathbb{R_+}$ of (\ref{bellman2}) for all $x\in\mathcal{I}$ within the class of PWC functions defined in (\ref{PWMLV}). The corresponding PWC control function $K_{r}:\mathcal{I}\mapsto \Gamma_r$ is given by%
\begin{equation}
K_{r}\left(  x\right)  =\arg\min_{r\in \Gamma_r}\max_{u\in U, \bar{x}\in T(x)}\hspace{-.05in}V\left(  A\bar{x}+Br-Bu\right)
,~\forall x\in \mathcal{I}.\label{GK}
\end{equation}

\noindent where $T(x)=T_i$ for $x \in T_i$. In the next subsection we show how to search and certify functions $V$ and $K_{r}$ satisfying (\ref{bellman2}) and (\ref{GK}).

%%%%%%%%%%%%%%%%%%%%%%%
\subsection{Searching for Numerical Solutions\label{Search4NumSol}}

The Bellman inequality $\left(  \ref{bellman2}\right)  $ is
solved via value iteration. The algorithm is initialized at $\Lambda
_{0}\left(  x\right)  =0$, for all $x\in\mathcal{T}$, and at stage $k+1$ it
computes a PWC function $\Lambda_{k+1}:\mathcal{T}%
\mapsto\mathbb{R}_{+}$ satisfying 
\begin{equation}
\Lambda_{k+1}\left(  x\right)  = \max\left\{  0,~\sigma_{\gamma}\left(  x\right)  +\min_{r\in\Gamma_{r}}%
\max_{u\in U, \bar{x}\in T(x)}\Lambda_{k}\left(  A\bar{x}+Br-Bu\right)  \right\}  . \label{Biterate}
\end{equation}

%\begin{multline}
%\Lambda_{k+1}\left(  x\right)  =\label{Biterate}\\
%\max\left\{  0,~\sigma_{\gamma}\left(  x\right)  +\min_{r\in\Gamma_{r}}%
%\max_{u\in U, \bar{x}\in T(x)}\Lambda_{k}\left(  A\bar{x}+Br-Bu\right)  \right\}  .
%\end{multline}

At each stage of the iteration, $\Lambda_{k+1}$ is computed and certified to satisfy \eqref{Biterate} for all $x \in \mathcal{T}$ as follows:
\begin{enumerate}
\item For every $i\in\left\{  1,2,\cdots,N_{T}\right\}$ and $j\in\left\{  1,2,\cdots,L\right\}$, define
\[
\sigma_{i} =\sup_{x\in T_{i}}\sigma_{\gamma}\left(  x\right),
\]
\[
Y_{ij} =\left\{  Ax+Br_{j}-Bu~|~x\in T_{i},~r_j\in  \Gamma_r, ~ u\in U\right\}  ,
\]

and find all the tiles that intersect with $Y_{ij}$
\[
\Theta_{ij}=\left\{ p~|~T_{p}\cap Y_{ij}\neq\left\{  \emptyset\right\}
,~p\in\left\{  1,2,\cdots,N_{T}\right\}  \right\}  .
\]

\item Let
\[
v_{s} =\Lambda_{k}\left(  x\right)  ,~x\in T_{s},~s\in\left\{  1,2,\cdots,N_{T}\right\}.\\
\]
Compute
\begin{equation*}
v_{ij} =\max_{s\in\Theta_{ij}}v_{s}.
\end{equation*}

\item For every tile $x\in T_{i} $ compute PWC functions:
\begin{align*}
\Lambda_{k+1}\left(x\right) & =\max\left\{  0,~\sigma_{i} +\min_{j}v_{ij}  \right\}.
\end{align*}
\end{enumerate}

When the iteration converges, it converges pointwise to a limit
$\Lambda:\mathcal{T}\mapsto\mathbb{R}_{\mathbb{+}},$ where the limit
satisfies, for all $x\in\mathcal{T},$ the equality
\begin{equation}
\Lambda\left(  x\right)  = \max\left\{  0,~\sigma_{\gamma}\left(  x\right)
+\min_{r\in\Gamma_{r}}\max_{u\in U, \bar{x}\in T(x)}\Lambda\left(  A\bar{x}+Br-Bu\right)  \right\}.
\label{Bell_Limit}
\end{equation}

%\Lambda\left(  x\right)  =\\
%\max\left\{  0,~\sigma_{\gamma}\left(  x\right)
%+\min_{r\in\Gamma_{r}}\max_{u\in U, \bar{x}\in T(x)}\Lambda\left(  A\bar{x}+Br-Bu\right)  \right\}.
%\label{Bell_Limit}
%\end{multline}

The largest $\gamma$ for which (\ref{Biterate}) converges is found through
line search. We take $V\left(  x\right)  =\Lambda\left(  x\right)  ,$\ for all
$x\in\mathcal{T}.$ The associated suboptimal control law is a PWC function defined over all tiles $T_i$ in the control invariant set $\mathcal{I}$ that satisfies \eqref{GK}.

%%%%%%%%%%%%%%%%%%%%%%%%%%%%%%%%%%%%%%%%%%%

\section{Theoretical Statements\label{TexSta}}

In this section, we show that under some technical assumptions, the value function in (\ref{bellman}) is zero beyond a bounded region. However, we first present a theorem that establishes the link between the full information feedback control problem and the Bellman inequality (\ref{bellman}). Note that in this section we use subscript notation for values of sequences at specific time instances instead of the bracket notion used elsewhere in the paper, that is $x_n$ is used in place of $x[n]$.
\bigskip

\begin{theorem}
\label{dpThmLB} Let $X$ be a metric space, $\Omega \subset \mathbb{R}$ be a compact metric space, $U \subset \mathbb{R}$ be a finite set, and $f:X \times \Omega \times U \mapsto X$ and $\sigma:X\mapsto\mathbb{R}$ be continuous functions. Then the following statements are equivalent:

\begin{enumerate}
\item[(i)] 
\begin{equation}
V_\infty(\bar{x})  \;  {\buildrel\rm def\over =}\; \sup_{\tau \in \mathbb{Z}_+} V_\tau(\bar{x}) < \infty, \quad \forall  \bar{x} \in X, \label{Vinf}
\end{equation}
\noindent where $V_\tau: X \mapsto \mathbb{R}_+ $ is defined by
\begin{equation}
V_\tau(\bar{x})= \max_{\theta_0}\min_{r_0} \max_{u_0, \theta_1} \cdots \min_{r_{\tau-2}} \max_{u_{\tau-2}, \theta_{\tau-1}} \sum^{\tau-1}_{n=0} h_{n+1} \sigma(x_n), \label{Vtau}
\end{equation}

\noindent with $r_n,~u_n,~ \theta_n$ restricted by $r_n \in \Omega,~ u_n \in U,~\theta_n \in \{0,1\}$ and $x_n,~h_n$ defined by
\begin{align}
x_{n+1} & =f(x_n,r_n,u_n),  \quad x_0 = \bar{x}, \quad \forall n \in \mathbb{Z}_+ \label{xn+1}\\
h_{n+1} & =\theta_n h_n,  \quad h_0 =1, \quad \forall n \in \mathbb{Z}_+.
\end{align}

\item[(ii)]The sequence of functions $\Lambda_{k}:X\mapsto\mathbb{R}_{+}$ defined by 
\begin{align}
\Lambda_{0}\left(  x\right)  & \equiv 0 \notag \\
\Lambda_{k+1}\left(  x\right)  & = \max\left\{  0,~\sigma\left(  x\right)  +\min_{r\in \Omega }\max_{u\in U}\Lambda_{k}\left(f(x,r,u)\right)  \right\} \label{FuzzyLamma}
\end{align}
converges pointwise to a limit $\Lambda_\infty: X\mapsto \mathbb{R}_{\mathbb{+}}.$

\item[(iii)] There exists a function $V:X \mapsto \mathbb{R}_{+}$ such that 
\begin{equation}
V(x)=\max \left\{0, ~\sigma(x) + \min_{r \in \Omega} \max_{u \in U} V(f(x,r,u))\right\} \label{BellmanEquality}
\end{equation}
for every $x\in X$.

\item[(iv)] There exists a function $V:X\mapsto\mathbb{R}_{+}$ such that 
\begin{equation}
V(x) \ge \sigma(x) + \min_{r \in \Omega}\max_{u \in U} V(f(x,r,u)), \quad \forall x\in X. \label{BelleIneq}
\end{equation}

\end{enumerate}

\noindent Moreover, if conditions (i)$-$(iv) hold, then $V_\infty$ is a solution of \eqref{BellmanEquality} and
\begin{align}
V_\infty = \Lambda_\infty & \ge V_k = \Lambda_k, \quad \forall k \in \mathbb{Z}_+ \label{bison}\\
V & \ge V_\infty. \label{buffalo}
\end{align}
for $V$ satisfying (iii). Furthermore,  for every $x_n$ satisfying \eqref{xn+1},
\begin{equation}
 \sup_{\tau} \sum_{n=0}^{\tau-1} \sigma(x_n) < \infty . \label{yak}
\end{equation}
\end{theorem}
\medskip

\begin{proof}
Please see the Appendix.
\end{proof}

\bigskip

%\begin{definition}
%The reachable set $\mathcal{R}$ of system \eqref{x+} is:
%\begin{multline}
%\mathcal{R} = \left \{x_t \in \mathbb{R}^m : x_{n+1}=Ax_n +Br_n -Bu_n,~ x_0=0, ~ \right.\\
%\left. r_n \in [-1,1], ~ u_n \in U,~ t>0\right\}.
%\end{multline}
%
%\end{definition}
%
%\bigskip

\begin{definition}
For $v \in \mathbb{R}^m\backslash \{ 0\}$, a \textit{cylinder with axis $v$} is a set of the form:
\begin{equation}
\mathcal{C}_{Q,\beta}(v) = \left\{ p\in \mathbb{R}^m :  \inf_{t \in \mathbb{R}} (p- tv)^TQ(p- tv) \le \beta \right\}
\end{equation}
\noindent where $Q \in \mathbb{R}^{m \times m}$, $Q = Q' > 0$, and $\beta > 0$.
\end{definition}
\medskip
\begin{remark}
A cylinder that is an invariant set for system \eqref{x+} is called an\textit{ invariant cylinder. }
\end{remark}
\bigskip

%Clearly, the invariant cylinder $\mathcal{C}_Q(e_1)$ of system \eqref{x+} contains its reachable set $\mathcal{R}$. 
%Furthermore, the set $E(e)=\{e'\zeta:\zeta \in \mathcal{C}_{Q,\beta}(e_1)\}$ is bounded for all $e \in \mathbb R^m$ orthogonal to $e_1$. Thus, 
The following theorem establishes that the value function is zero for all $x$ outside a certain bounded region.
\bigskip

%Since matrix $A$ has exactly one eigenvalue on the unit circle, with all other eigenvalues strictly inside the unit circle, the set $E(e)=\{e'\zeta:\zeta \in \mathcal{R}\}$ is bounded for every $e \ne ce_1$, where $c \in \mathbb{R}\backslash\{0\}$, and unbounded otherwise. That is, there exists a $\tilde{\beta} \in \mathbb{R}$ such that $\mathcal{R} \subset \mathcal{C}_{\tilde{\beta}}(e_1)$. 

\begin{theorem}
\label{boundednessTHM}
Let $U \subset \mathbb{R}$ be a fixed finite set. Consider the system defined by equation \eqref{x+}, where $x \in \ell_+(\mathbb{R}^m), ~r \in \ell_+([-1,1]), ~u \in \ell_+(U)$, and the pair $(A,B)$ is controllable. Suppose that $A$ has exactly one eigenvalue on the unit circle. Let $e_{1}$ denote the eigenvector corresponding 
to the eigenvalue of $A$ that is on the unit circle. Let  $\beta>0$ and $Q \in \mathbb{R}^{m \times m}$, $Q = Q' > 0$ be such that $\mathcal{C}_{Q,\beta}(e_1)$ is an invariant cylinder for system \eqref{x+}. Let $V$ be defined by \eqref{Vinf} and $\sigma$ be BIBO. If the set 
\begin{equation}
S_0=\{x\in \mathcal{C}_{Q,\beta}(e_1) :\sigma(x)>-\epsilon_0\}
\end{equation}
is bounded for some $\epsilon_0 > 0$, then the set 
\begin{equation}
M = \{x \in \mathcal{C}_{Q,\beta}(e_1):V(x)\neq0\} \label{Vne0Bounded}
\end{equation}
is also bounded.
\end{theorem}

\begin{proof}
Please see the Appendix.
\end{proof}

%%Observability of $(C,A)$ guarantees that for every eigenvector $e_{i}$ of
%%$A\in\mathbb{R}^{m\times m}$, $Ce_{i}\neq0$ for all $i\in\{1,~\cdots,~m\}$.
%%Thus, $\phi(q)\neq0$ and consequently $\sigma(x) \neq 0$ for every $x$ with a nonzero component in the direction
%%of $e_{1}$.
%%

%%\begin{definition}
%%Let $\mathcal{S}(A,B,C,x_0)$ denote the set of all quadruplet of sequences $(x,r,u,q)$, where $x:\mathbb{Z}_+ \mapsto \mathbb{R}^m, ~r:\mathbb{Z}_+ \mapsto \mathbb{R}, ~u:\mathbb{Z}_+ \mapsto \mathbb{R}, ~q:\mathbb{Z}_+ \mapsto \mathbb{R}$ satisfying a controllable, minimum-phase LTI system described by
%%\begin{align}
%%x[n+1] &= A x[n] + Br[n]-Bu[n], \quad x[0] = x_0 \label{x+2}\\
%%q[n] &= Cx[n]. \label{q2} 
%%\end{align}
%%\end{definition}

%%%%%%%%%%%%%%%%%%%%%%%%%%%%%%%%%%%%%%%%%%%%%%%%%%%%%%%%%%%%%%%%%%%%%%%%%%%%%%%%
\section{NUMERICAL EXAMPLE\label{NumEx}}

Consider the example in \cite{PhDBib:Mitra1}, were the dynamical system $L_{G}$ \eqref{LG0} has transfer function
\[
H(z)=\frac{z+1}{z(z-1)}.
\]

%and state space representation
%\begin{align}
%x_1[n+1] &= x_1[n] + r[n] - u[n] \\
%x_2[n+1] &= x_1[n] \\
%q[n] &= x_1[n] + x_2[n].
%\end{align}
%	\[
%	A=\left[
%	\begin{array}
%	[c]{cc}%
%	1 & 0\\
%	1 & 0
%	\end{array}
%	\right]  ,~B=\left[
%	\begin{array}
%	[c]{c}%
%	1\\
%	0
%	\end{array}
%	\right]  ,~C=\left[
%	\begin{array}
%	[c]{cc}%
%	1 & 1
%	\end{array}
%	\right]  .
%	\]

Let $U=\left\{  -1.5,0,1.5\right\}$, $\phi(x)=\left\vert
Cx\right\vert $, and $x= \begin{bmatrix} x_1 & x_2\end{bmatrix}^T.$ From \cite{PhDBib:Mitra1}, the strong invariant set $\mathcal{I}$ is given by \vspace{-.05in}
\begin{equation}
\mathcal{I} =\{x\in \mathbb{R}^2~ : ~\left\vert x_1 - x_2 \right\vert \leq2.5\}.\label{zedonk}%
\end{equation}
Due to the pole at $z=1$, the strong invariant set $\mathcal{I}$ given by \eqref{zedonk} is unbounded and defines an infinite strip in $\mathbb{R}^2$. However, according to Theorem \ref{boundednessTHM} we need to search for $V(x)$ only inside a bounded region within this infinite strip, since $V(x)=0$ for all $x$ outside a certain bounded region. The bounded region is found via trial and error. We select a grid spacing of $\Delta=1/64$. Following the procedures outlined in subsections \ref{NumSolnBell} and \ref{Search4NumSol}, the
largest $\gamma$ for which the iteration in (\ref{Biterate}) converges to the
limit $\Lambda$ in (\ref{Bell_Limit}), is $\gamma=0.925$, which is a certified lower bound on the performance of any arbitrary ADC with respect to the specific performance measure selected. Figures \ref{VLB}, \ref{Vslice}, and \ref{levelSet0} show the value function $V$, the cross section of $V$%the value function along the line $x_1=-x_2$
, and the zero level set of $V$%the value function
, respectively. Figures \ref{Kr} and \ref{SliceKr} show the control function and its cross section%the cross section of the control function along the line $x_1=x_2$
, respectively. The certified upper bound for the performance of the ADC designed in \cite{PhDBib:Mitra1} with respect to the same performance criteria is 1.1875. 

%Fig. \ref{SliceKr} shows the cross section of the control function along the line $x_1=x_2$. Utilizing the optimized feedback law for the input of the ADC $K_r(x)$ in conjunction with the optimized control law for the output of the ADC $K(x,r)$, from [19], with initial condition $x[0]=0$, yields a specific state trajectory and an associated worst case input sequence depicted in Fig. \ref{Badr}.
%%%%% %%%%%%It is interesting to note that the value function for the lower bound is very different from the value function for the upper bound (Fig. \ref{VUB}). 

\begin{figure}[h]
\centering
\includegraphics[scale=.35]{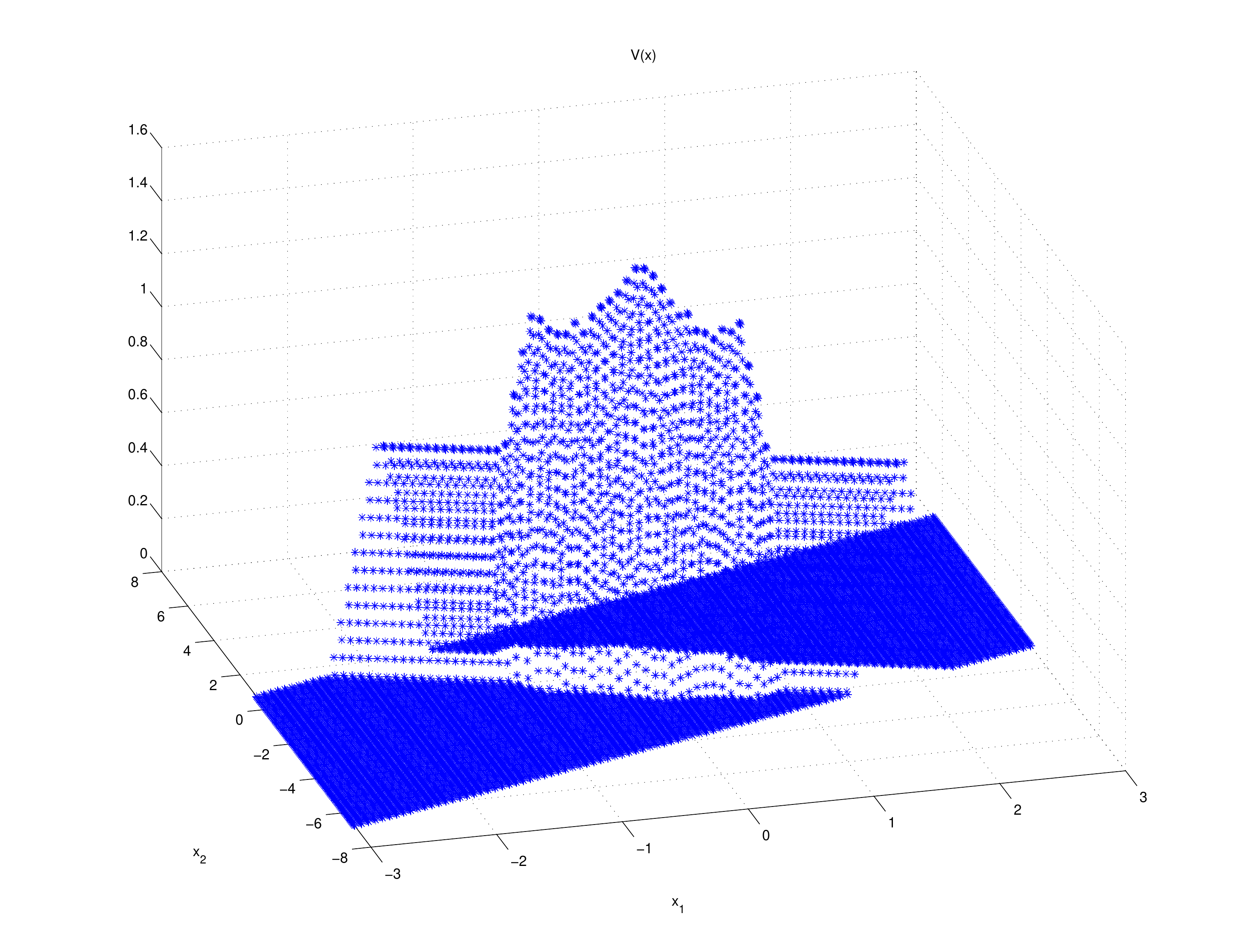}
\caption{Value Function $V(x)$ for Lower Bound}
\label{VLB}
\end{figure}

\begin{figure}[h]
\subfloat[Cross Section of $V(x)$ along $x_1 = -x_2$]{\includegraphics[scale=.18]{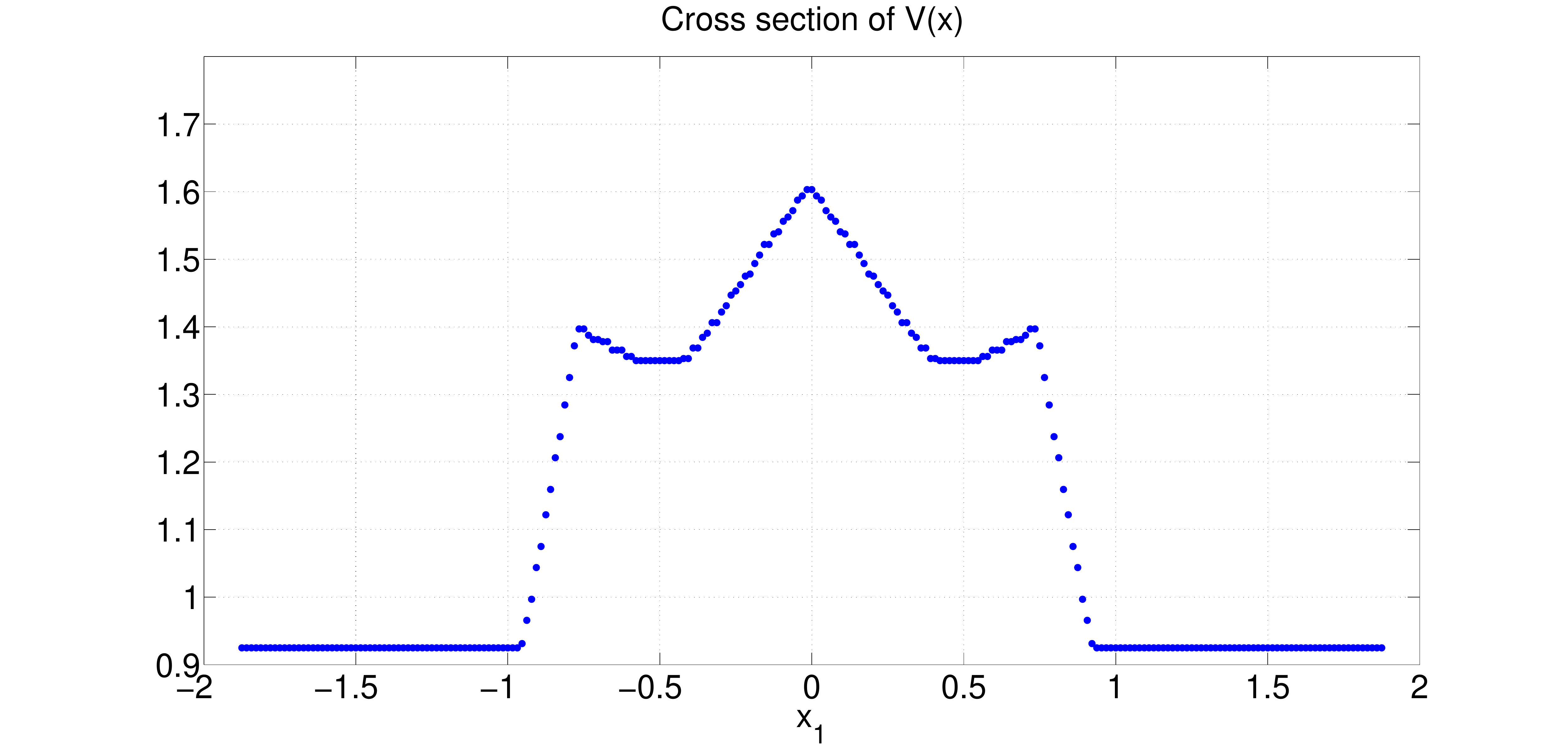} \label{Vslice}}
\subfloat[Zero Level Set of $V(x)$] {\includegraphics[scale=.135]{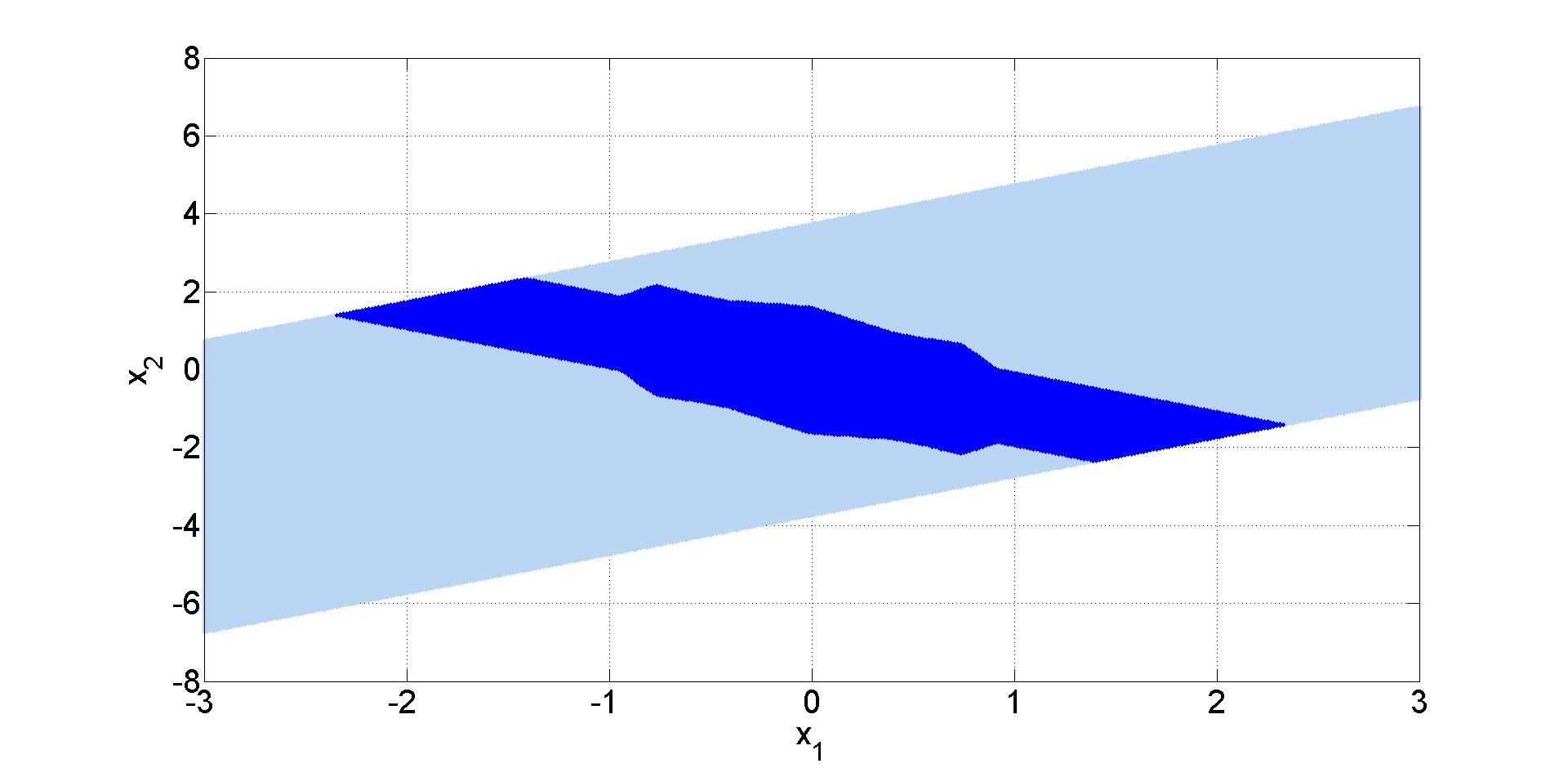} \label{levelSet0}}
\caption{Value Function $V(x)$}
\end{figure}

%\begin{figure}[h]
%\vspace{.15in}
%\includegraphics[scale=.145]{CrossSectionV.pdf}   
%\caption{Cross Section of $V(x)$ along $x_1 = -x_2$}
%\label{Vslice}
%\end{figure}
%
%\begin{figure}[h]
%
%\includegraphics[scale=.11]{ZeroLevelSetV2.jpg}   
%\caption{Zero Level Set of $V(x)$}
%\label{levelSet0}
%\end{figure}

%\begin{figure}[h]
%\includegraphics[scale=.24]{K_r.pdf}   
%\caption{Control $K_r(x)$ for Lower Bound}
%\label{Kr}
%\vspace{.1in}
%\includegraphics[scale=.15]{CrossSectionKr.pdf}   
%\caption{Cross Section of Control $K_r(x)$ along $x_1=x_2$}
%\label{SliceKr}
%\end{figure}

\begin{figure}[h]
\subfloat[Control $K_r(x)$ for Lower Bound]{\includegraphics[scale=.3]{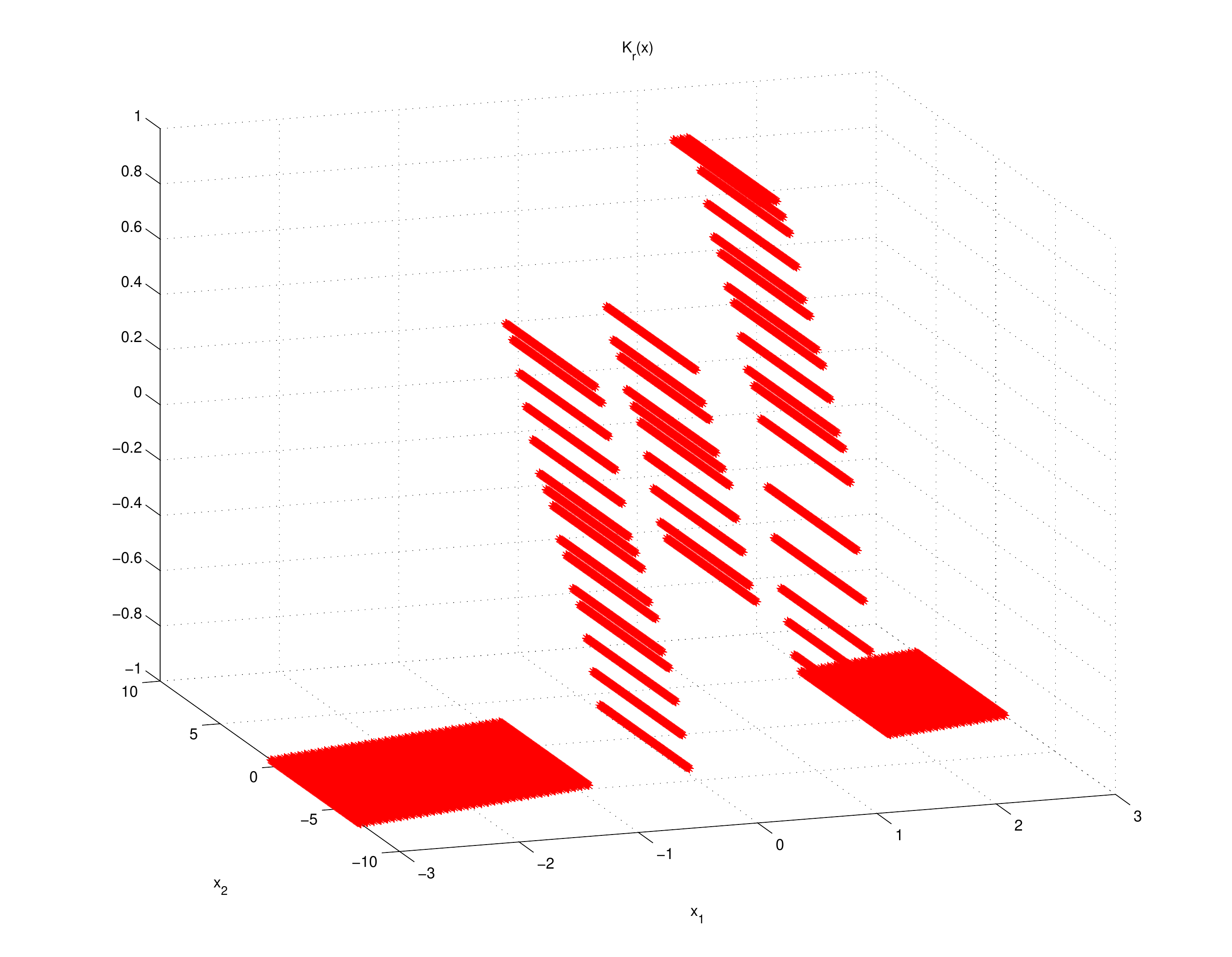}   \label{Kr}}
\subfloat[Cross Section of Control $K_r(x)$ along $x_1=x_2$]{\includegraphics[scale=.2]{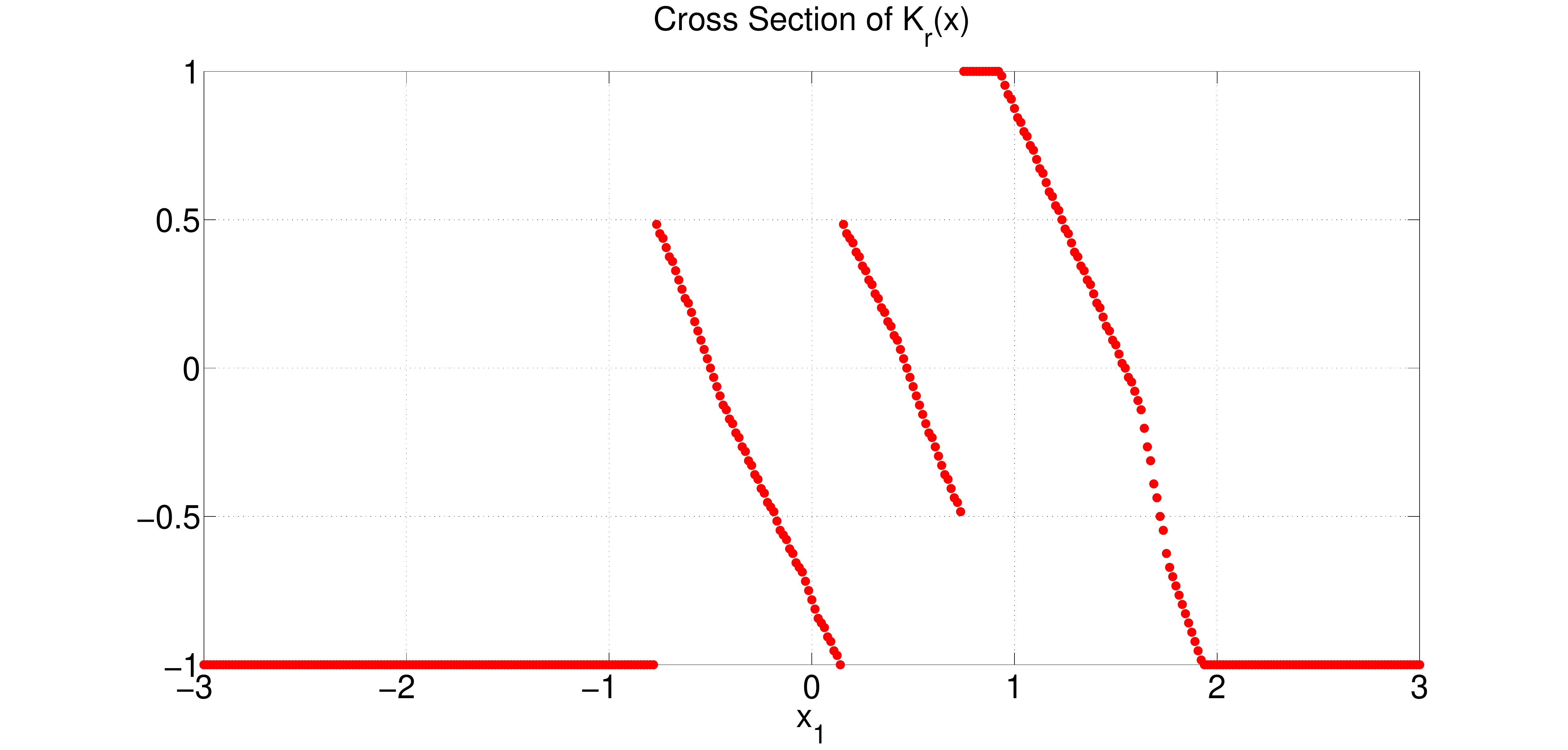}   \label{SliceKr}}
\caption{Control Function $K_r(x)$}
\end{figure}

%
%
%\begin{figure}[h]
%\includegraphics[scale=.23]{Bad_r.eps}   
%\caption{Worst Case Input Sequence r}
%\label{Badr}
%\end{figure}

%
%\begin{figure}[h]
%
%\includegraphics[scale=.5]{V_UB.eps}  
%\caption{Value Function $V(x)$ for Upper Bound}
%\label{VUB}
%\end{figure}
%
%\begin{figure}[h]
%
%\includegraphics[scale=.5]{U_UB_r_eq0.eps}   
%\caption{Control $K(x,r)$ for $r=0$ for Upper Bound}
%\label{Kxr}
%\end{figure}

\bigskip 
%%%%%%%%%%%%%%%%%%%%%%%%%%%%%%%%%%%%%%%%%%%%%%%%%%%%%%%%%%%%%%%%%%%%%%%%%%%%%%%%

\section{CONCLUSION\label{Future}}
In this paper, we studied performance limitations of Analog to Digital Converters (ADCs). The performance of an ADC was defined in terms of a measure that represents the worst case average intensity of the filtered input matching error. The passband of the shaping filter defines the frequency region in which the error is to be minimized. The problem of finding a lower bound for the performance of an ADC was associated with a full information feedback optimal control problem and formulated as a dynamic game in which the input of the ADC (control variable) played against the output of the ADC (quantized disturbance). Since the disturbances can exceed the control variable in magnitude, if the shaping filter has a pole on the unit circle, then there does not exist a bounded control invariant set, which presents a challenge for numerical computations. This challenge is overcome with theoretical results that show that the value function is zero beyond a bounded region, thus computations need to be done only over this bounded region. A numerical algorithm was presented that provided certified solutions to the underlying Bellman inequality in parallel with the control law; hence, certified lower bounds on the performance of arbitrary ADCs with respect to the adopted performance criteria. 

%%%%%%%%%%%%%%%%%%%%%%%%%%%%%%%%%%%%%%%%%%%%%%%%%%%%%%%%%%%%%%%%%%%%%%%%%%%%

\section{APPENDIX\label{APP}}

\begin{observation}\label{monotonicity}
The sequence of functions $\Lambda_k:X\mapsto \mathbb R_+$ defined by \eqref{FuzzyLamma} is monotonically increasing. 
\end{observation}
\bigskip
\begin{proof}
The proof is done by induction. Since $\Lambda_0(x) =0$ for all $x \in X$, it follows that
\begin{equation}
\Lambda_1(x) = \max\left\{0, ~\sigma(x) \right\} \ge \Lambda_0(x), \qquad \forall x \in X.
\end{equation}
Assume, $\Lambda_k(x) \ge \Lambda_{k-1}(x)$ for all $x \in X$. This assumption in conjunction with equation \eqref{FuzzyLamma}, results in the following inequality
\begin{align*}
\Lambda_{k+1}\left(  x\right) & \ge \max\left\{  0,~\sigma\left(  x\right)  +\min_{r\in \Omega }\max_{u\in U}\Lambda_{k-1}\left(f(x,r,u)\right)  \right\} \\
                                                & = \Lambda_k(x).
\end{align*}
Therefore,
\begin{equation*}
\Lambda_0(x) \le \Lambda_1(x) \le \cdots \le \Lambda_k(x), \quad \forall x \in X,~ k \in \mathbb{Z}_+  
\end{equation*}
\end{proof}

\subsubsection*{\bf Proof of Theorem \ref{dpThmLB}}

{\bf (i)$\implies$(ii)} 
For $\tau = 0$, equation \eqref{Vtau} simplifies to:
\begin{equation}
V_0(x_0) = 0, \quad \forall x_0 \in X.
\end{equation}
For $\tau = 1$, we have: 
\begin{equation*}
V_1(x_0) = \max \{0, \sigma(x_0)\}.
\end{equation*}
The rest of the proof is done by induction over $\tau$. For $\tau = 2$, we have: 
\begin{align*}
V_2(x_0) & = \max \{0, \sigma(x_0)+ \min_{r_0}\max_{u_0,\theta_1}\theta_1 \sigma(x_1)\} \\
		& = \max \{0, \sigma(x_0)+ \min_{r_0}\max_{u_0}V_1(f(x_0,r_0,u_0))\}.
\end{align*}
Assume that, 
\begin{equation*}
V_k(x_0) = \max \{0, \sigma(x_0)+ \min_{r_0}\max_{u_0}V_{k-1}(f(x_0,r_0,u_0))\}.
\end{equation*}

Define $\widetilde{h}_{n+1} = \theta_n \widetilde{h}_n$, $\widetilde{h}_1=1$, for $n=1,2, 3 \cdots$. Equation \eqref{Vtau} can be equivalently written for $\tau = k+1$ as follows:
\begin{equation*}
V_{k+1}(x_0)  =   \max_{\theta_0}\Bigg[ \theta_0 \Big(\sigma(x_0)+   
 \min_{r_0} \max_{u_0}\underbrace{\max_{\theta_1} \cdots \min_{r_{k-1}} \max_{u_{k-1}, \theta_{k}} \sum^{k}_{n=1} \widetilde{h}_{n+1} \sigma(x_n)}_{V_k(x_1)} \Big) \Bigg].
\end{equation*}

%\begin{multline*}
%V_{k+1}(x_0)  =   \max_{\theta_0}\Bigg[ \theta_0 \Big(\sigma(x_0)+   \\
% \hspace{.2in}\min_{r_0} \max_{u_0}\underbrace{\max_{\theta_1} \cdots \min_{r_{k-1}} \max_{u_{k-1}, \theta_{k}} \sum^{k}_{n=1} \widetilde{h}_{n+1} \sigma(x_n)}_{V_k(x_1)} \Big) \Bigg].
%\end{multline*}
Therefore,
\begin{equation}
V_{k+1}(x_0) = \max \{0, \sigma(x_0)+ \min_{r_0}\max_{u_0}V_{k}(f(x_0,r_0,u_0))\}. \label{Vk+1}
\end{equation}

From Observation \ref{monotonicity}, we know that the sequence of functions $V_k$ is monotonically increasing. 
%
%We show by induction that the sequence of functions $V_k$ is monotonically increasing. Since $V_0(x) =0$ for all $x \in X$, it follows that
%\begin{equation}
%V_1(x) = \max\left\{0, ~\sigma(x) \right\} \ge V_0(x), \qquad \forall x \in X.
%\end{equation}
%Assume, $V_k(x) \ge V_{k-1}(x)$ for all $x \in X$. This assumption in conjunction with equation \eqref{Vk+1}, results in the following inequalities
%\begin{align*}
%V_{k+1}\left(  x\right) & \ge \max\left\{  0,~\sigma\left(  x\right)  +\min_{r\in \Omega }\max_{u\in U}V_{k-1}\left(f(x,r,u)\right)  \right\} \\
%                                                & = V_k(x).
%\end{align*}
%Therefore,
%\begin{equation*}
%V_0(x) \le V_1(x) \le \cdots \le V_k(x), \quad \forall x \in X,~ k \in \mathbb{Z}_+  
%\end{equation*}
Since a monotonic sequence of functions converge if and only if it is bounded, we have convergence of the sequence.

{\bf(ii)$\implies$(i)}
Again from Observation \ref{monotonicity}, the sequence of functions $V_k$ is monotonically increasing; thus, in conjunction with convergence of the sequence, we have $\Lambda_\infty < \infty$. Furthermore, since $\Lambda_k(x)  \ge 0$ for all $k \in \mathbb{Z}_+$, we also have $\Lambda_\infty(x)  \ge 0$. It only remains to show that equation \eqref{FuzzyLamma} is equivalent to \eqref{Vtau}. Equation \eqref{FuzzyLamma} for $k=1$ is trivially equivalent to:
\begin{equation*}
\Lambda_1(x_0) = \max_{\theta_0\in\{0,1\}} \theta_0 \sigma(x_0) , \quad \forall x_0 \in X
\end{equation*}
For $k=2$, equation \eqref{FuzzyLamma}  is equivalent to:
\begin{align*}
\Lambda_2(x_0) & =  \max \left \{0,~  \sigma_0(x_0) + \min_{r_0} \max_{u_0} \Lambda_1(x_1)  \right\}\\
& = \max_{\theta_0} \left [ \theta_0  \left(  \sigma_0(x_0) + \min_{r_0} \max_{u_0,\theta_1}\{ \theta_1 \sigma(x_1) \} \right)  \right]\\
& = \max_{\theta_0} \min_{r_0} \max_{u_0,\theta_1} \sum_{n=0}^{1}h_{n+1}\sigma(x_n).
\end{align*}
Assume,
\begin{equation*}
\Lambda_k(x_0)= \max_{\theta_0}\min_{r_0} \max_{u_0, \theta_1} \cdots \min_{r_{k-2}} \max_{u_{k-2}, \theta_{k-1}} \sum^{k-1}_{n=0} h_{n+1} \sigma(x_n). 
\end{equation*}
Substituting the above equation into \eqref{FuzzyLamma} we have:
%\begin{multline*}
%\Lambda_{k+1}(x_0) = \max_{\theta_0} \Bigg[\theta_0 \Big( \sigma(x_0) + \\
%                                           \min_{r_0} \max_{u_0} \underbrace{\max_{\theta_1}  \cdots \min_{r_{k-1}} \max_{u_{k-1}, \theta_{k}} \sum^{k}_{n=1} \widetilde{h}_{n+1} \sigma(x_n)}_{\Lambda_k(x_1)} \Big) \Bigg],
%\end{multline*}
\begin{equation*}
\Lambda_{k+1}(x_0) = \max_{\theta_0} \Bigg[\theta_0 \Big( \sigma(x_0) + 
                                           \min_{r_0} \max_{u_0} \underbrace{\max_{\theta_1}  \cdots \min_{r_{k-1}} \max_{u_{k-1}, \theta_{k}} \sum^{k}_{n=1} \widetilde{h}_{n+1} \sigma(x_n)}_{\Lambda_k(x_1)} \Big) \Bigg],
\end{equation*}
which is equivalent to \eqref{Vtau}. Finally, since the sequence of functions $\Lambda_k$ is monotonically increasing, the limit as $k \to \infty$ of $\Lambda_k$ is equivalent to its supremum over $k$.

{\bf(i)$\implies$(iii)}
Substituting \eqref{Vk+1} into \eqref{Vinf} and interchanging the order of the supremum over $\tau$ with the maximum, we get 
\begin{equation*}
V_\infty(\bar{x}) = \max \{0, \sigma(x_0)+ \sup_{\tau \in \mathbb{Z}_+} \min_{r_0}\max_{u_0}V_{\tau-1}(f(x_0,r_0,u_0))\}.
\end{equation*}

As discussed in the proof of (ii)$\implies$(i), the supremum over $\tau$, in the expression above, is equal to the limit as $\tau \to \infty$. Moreover, a well-known theorem from Analysis states that: given a metric space $X$, a compact metric space $\Omega$, and a continuous function $g:X \times \Omega \mapsto \mathbb R$, the function $\hat g: X \mapsto \mathbb R$ defined by 
\[
\hat g(x) = \max_{r \in \Omega} g(x,r), \quad \text{or} \quad \hat g(x) = \min_{r \in \Omega} g(x,r)
\] 
is continuous. Furthermore, given a compact metric space $\Omega$, and a monotonically increasing sequence of continuous functions $g_k:\Omega \mapsto \mathbb R$ such that $\lim_{k  \to \infty}g_k(r)$ is finite for every $r \in \Omega$, the following equality holds: 
\[ 
\lim_{k \to \infty} \min_{r \in \Omega} g_k(r) = \min_{r \in \Omega} \lim_{k \to \infty} g_k(r)
\]

Therefore, we have \eqref{BellmanEquality}.

%Furthermore, compactness of $R$, finiteness of $U$, and continuity of functions $f$ and $\sigma$, result in the continuity of the sequence of functions $V_k$ for all $k$. Consequently, the order of the limit and minimum can be interchanged, since we have  a monotonically increasing sequence of continuous functions over a compact metric space $R$. Therefore, we have \eqref{BellmanEquality}.

%%%---------- Lemma --------------------
%\begin{lemma} \label{continuity_of_V}
% Let $X$ be a metric space, $R$ be a compact metric space, $U$ be a finite set, and $f:X \times R \times U \mapsto X$ and $\sigma: X \mapsto \mathbb{R}$ be continuous functions. Let  $\Lambda_k:X\mapsto \mathbb R_+$ be a sequence of functions defined by \eqref{FuzzyLamma}, then the functions $\{\Lambda_k\}$ are continuous.
%\end{lemma}
%\begin{proof}
%By standard methods from analysis.
%\end{proof}
%\bigskip
%\begin{lemma}\label{swap_limit_inf}
%Given a compact metric space $R \subset \mathbb{R}$ and a monotonically increasing sequence of continuous functions $g_k: X \mapsto \mathbb{R}$, then
%\begin{equation}
%\lim_{k \rightarrow \infty} \inf_{r \in R} g_k(r) = \inf_{r \in R} \lim_{k \rightarrow \infty} g_k (r)
%\end{equation}
%\end{lemma}
%\smallskip
%\begin{proof}
%By standard methods from analysis.
%\end{proof}
%\bigskip
%%----------------------------------------------

{\bf(iii)$\implies$(iv)}
Trivially true.

{\bf(iv)$\implies$(i)}
Since $V(x) \ge 0$ for all $x$, we can rewrite \eqref{BelleIneq} as
\begin{equation}
V(x) \ge \max \left \{0, ~\sigma(x) + \min_{r}\max_{u} V(f(x,r,u))\right \}. \label{BelleIneq2}
\end{equation}

Inequality \eqref{BelleIneq2} holds for all $x \in X$, thus it holds for the sequence $\{x_0, x_1, \cdots, x_{k-1}\}$, where $x_{k-1}$ satisfies \eqref{xn+1}. Now take inequality \eqref{BelleIneq2} with $x$ replaced by $x_0$ and substitute for $V(f(x_0,r_0,u_0))$ the corresponding inequality for $V(x_1)$: 
%\begin{multline*}
%V(x_0) \ge \max \Bigg\{0, ~\sigma(x_0) + \\ 
% \hspace{.25in}           \left . \min_{r_0}\max_{u_0}\max \left[0,~ \sigma(x_1) + \min_{r_1} \max_{u_1} V(f(x_1,r_1,u_1)) \right] \right \}. 
%\end{multline*}
\begin{equation*}
V(x_0) \ge \max \left\{0, ~\sigma(x_0) + 
            \min_{r_0}\max_{u_0}\max \left[0,~ \sigma(x_1) + \min_{r_1} \max_{u_1} V(f(x_1,r_1,u_1)) \right] \right \}. 
\end{equation*}
Equivalently,
%\begin{multline*}
%V(x_0) \ge \max_{\theta_0} \theta_0 \Bigg(\sigma(x_0) + \\ 
%\hspace{.25in}                                       \left . \min_{r_0}\max_{u_0,\theta_1} \theta_1 \left[\sigma(x_1) + \min_{r_1} \max_{u_1} V(f(x_1,r_1,u_1)) \right] \right). 
%\end{multline*}
\begin{equation*}
V(x_0) \ge \max_{\theta_0} \theta_0 \left(\sigma(x_0) + 
                                      \min_{r_0}\max_{u_0,\theta_1} \theta_1 \left[\sigma(x_1) + \min_{r_1} \max_{u_1} V(f(x_1,r_1,u_1)) \right] \right). 
\end{equation*}

Repeating this process, we have:
%\begin{multline*}
%V(x_0) \ge \overbrace{\max_{\theta_0} \min_{r_0}\max_{u_0,\theta_1} \cdots \min_{r_{k-2}}\max_{u_{k-2},\theta_{k-1}}\Bigg[\sum_{n=0}^{k-1}h_{n+1}\sigma(x_n)}^{S(x_0)} \\
%                                               \left . + \min_{r_{k-1}} \max_{u_{k-1}} h_nV(f(x_{k-1},r_{k-1},u_{k-1})) \right]. 
%\end{multline*}
\begin{equation*}
V(x_0) \ge \overbrace{\max_{\theta_0} \min_{r_0}\max_{u_0,\theta_1} \cdots \min_{r_{k-2}}\max_{u_{k-2},\theta_{k-1}}\Bigg[\sum_{n=0}^{k-1}h_{n+1}\sigma(x_n)}^{S(x_0)} 
                                                + \min_{r_{k-1}} \max_{u_{k-1}} h_nV(f(x_{k-1},r_{k-1},u_{k-1})) \Bigg]. 
\end{equation*}
After rearranging terms we have:
%\begin{multline*}
%S(x_0)   \le V(x_0) - \\
%            \max_{\theta_0} \min_{r_0}\max_{u_0,\theta_1} \cdots \min_{r_{k-1}} \max_{u_{k-1}} h_nV(x_k) . 
%\end{multline*}
\begin{equation*}
S(x_0)   \le V(x_0) - 
            \max_{\theta_0} \min_{r_0}\max_{u_0,\theta_1} \cdots \min_{r_{k-1}} \max_{u_{k-1}} h_nV(x_k) . 
\end{equation*}
Non-negativity and existence of $V$ guarantees:
\begin{equation}
S(x_0) \le V(x_0) < \infty.\label{zebrasMissMe}
\end{equation}
Since \eqref{zebrasMissMe} holds for all $k$, we have \eqref{Vinf}.
\bigskip

The proof for \eqref{bison}-\eqref{yak} is as follows:
Substituting \eqref{Vinf} into the right hand side of \eqref{BellmanEquality} and using the reasoning in (i) $\implies$ (iii) it is easy to see that \eqref{Vinf} is a solution of \eqref{BellmanEquality}. Furthermore, \eqref{bison} was proved within the proof of (i) $\implies$ (ii). Inequality \eqref{buffalo} states that \eqref{Vinf} is the minimal solution of \eqref{BellmanEquality}, this is proven by induction. Let $F$ be a function that maps function $V$ on $X$ into function $FV$ on $X$, defined according to:
\begin{equation*}
(FV)(x)=\max \left\{0, ~\sigma(x) + \min_{r \in \Omega} \max_{u \in U} V(f(x,r,u))\right\}, 
\end{equation*}

then $V=FV$. Since $V \ge 0$, we have $V \ge V_0$. Assume $V \ge V_k$, and apply mapping $F$ to both sides to get
\begin{equation*}
FV \ge FV_k = V_{k+1}.
\end{equation*}
Therefore, $V \ge V_k$ for all $k$, and thus \eqref{Vinf} is the minimal solution of \eqref{BellmanEquality}. Finally, \eqref{yak} is obtained by substituting into \eqref{Vinf} the argument of minimums and maximums for the sequences $r$, $u$, and $\theta$ respectively.
 \begin{flushright}
$\blacksquare$
\end{flushright}

%
%
%\begin{lemma}l
%Given a compact metric space $\Omega \subset \mathbb{R}$ and a monotonically increasing sequence of continuous functions $f_k: \Omega \mapsto \mathbb{R}$, then
%\begin{equation}
%\lim_{n \rightarrow \infty} \inf_{r \in \Omega} f_n(r) = \inf_{r \in \Omega} \lim_{n \rightarrow \infty} f_n (r)
%\end{equation}
%\end{lemma}

The proof of Theorem \ref{boundednessTHM} relies on Lemmas \ref{finiteX} and \ref{supVexists} below. 

\bigskip

\begin{lemma}\label{finiteX}
Let $(A,B)$ be a controllable pair. Then, for every bounded set $\Xi\subset\mathbb{R}^{m}$, there exists a finite set $\widetilde{X}\subset \mathbb{R}^m$ and a function $\rho:\Xi\mapsto[-1,1]^{m}$, such that $x_m \in\widetilde{X}$ whenever $x_0 \in \Xi$ for every solution $(x,r)$ of
\begin{align}
x_{n+1}  & =Ax_n+Br_n  -Bu_n,\quad n\leq m \label{stateEqn}\\
r_n    & =\rho(x_{0})_{n}, \quad n\leq m \label{r_func}
\end{align}
for every $u \in \ell_+(U)$, where $\rho(x_{0})_{n}$ denotes the $n$-th element of $\rho(x_{0}).$
\end{lemma}

\begin{proof}
The solution to \eqref{stateEqn} is given by
\begin{equation}
x_m=A^{m}x_{0}+\displaystyle\sum_{i=0}^{m-1}A^{i}Br_{m-i-1}-\displaystyle\sum_{i=0}^{m-1}A^{i}Bu_{m-i-1}. \label{dolphin}
\end{equation}

Since $\Xi$ is bounded, $A^m \Xi$ is also bounded, thus the first term on the right hand side of \eqref{dolphin} is bounded. Let $L_c$ denote the controllability matrix:
\begin{equation*}
L_c = [A^{m-1}B \cdots AB ~~B].
\end{equation*} 

Construct a finite set $\widetilde{\Xi}_F \subset \widetilde{\Xi}$ as follows: let $\widetilde{\Xi}_F$ be the intersection of $\widetilde{\Xi}$ and the set consisting of uniformly spaced Cartesian grid points with spacing $\Delta$, where
\begin{equation*}
\Delta \le 2/ ||L_c^{-1}||_\infty.
\end{equation*}

Then for every $y_0 \in \widetilde{\Xi}$, there exists $\tilde{\xi} \in \widetilde{\Xi}_F$ such that:
\begin{equation*}
||y_0 - \tilde{\xi}||_\infty \le \Delta/2
\end{equation*}

Thus, 
\begin{equation*}
||L_c^{-1}(y_0 - \tilde{\xi})||_\infty \le 1,
\end{equation*}
which implies

\begin{equation}
L_c^{-1}(A^{m}x_{0}-\tilde{\xi}) \in [-1,1]^m, \quad \forall x_{0}\in\Xi \label{finiteSetCondition}
\end{equation}

Then for
\begin{equation}
\rho(x_{0})=-L_c^{-1}(A^{m}x_{0}-\tilde{\xi}),\label{badR}%
\end{equation}
we have
\begin{equation*}
x_m=\tilde{\xi}-\displaystyle\sum_{i=0}^{m-1}A^{i}Bu_{m-i-1}.
\end{equation*}

Since  $\widetilde{\Xi}_F$ and $U$ are finite sets, $x_m$ takes only a finite number of values.
\end{proof}

\bigskip

\begin{lemma} \label{supVexists}
Assume $(A,B)$ is controllable and the function $\sigma:\mathbb{R}^{m}
\mapsto\mathbb{R}$ is BIBO. If %the minimal solution of
%\begin{equation}
% V(x)=\max \left \{0,~ \sigma(x)+ \min_{r \in [-1,1]} \max_{u \in U} V(Ax+Br-Bu)\right \} \label{belleEQ}
%\end{equation}
%given by 
$V_\infty$ in \eqref{Vinf}-\eqref{Vtau} satisfies
\begin{equation}
V_\infty(x)<\infty,\quad\forall x\in\mathbb{R}^{m}\label{Vbounded}
\end{equation}
then $V_\infty$ is BIBO.
 \end{lemma}

\begin{proof}
Let $\alpha: \mathbb{R}^m \times \ell_+([-1,1]) \times \ell_+(U) \times \mathbb{Z}_+ \mapsto \mathbb{R}^m$ be a function that
maps the initial state $x_0$ and sequences $r$ and $u$ to the state $x$ at time $k$, where the evolution of the state is given by $x[n+1]=Ax[n]+Br[n]-Bu[n]$:
\begin{equation}
\alpha(x_0,r,u,k) = x_k. \label{alpha_x}
\end{equation}

%From Theorem \ref{dpThmLB}, if (\ref{Vbounded}) holds, then for every initial condition $x_{0}$ the function in \eqref{Vinf} is a minimal solution of (\ref{belleEQ}). 
Equation \eqref{Vtau} can be equivalently written as:

%\begin{multline}
% V_\tau(x_0)   = \max_{\theta_0} \min_{r_0} \max_{u_0, \theta_1} \cdots \min_{r_{m-1}} \max_{u_{m-1}}\left \{ \sum^{m-1}_{n=0} h_{n+1} \sigma(x_n)\right.\\
%                             \left. +\max_{ \theta_{m}}\min_{r_{m}} \max_{u_{m}, \theta_{m+1}} \cdots \min_{r_{\tau-2}} \max_{u_{\tau-2}, \theta_{\tau-1}} \sum^{\tau-1}_{n=m}h_{n+1}  \sigma(x_n)\right\} \label{22sums}
%\end{multline}
\begin{equation}
 V_\tau(x_0)   = \max_{\theta_0} \min_{r_0} \max_{u_0, \theta_1} \cdots \min_{r_{m-1}} \max_{u_{m-1}}\left \{ \sum^{m-1}_{n=0} h_{n+1} \sigma(x_n)\right.
                             \left. +\max_{ \theta_{m}}\min_{r_{m}} \max_{u_{m}, \theta_{m+1}} \cdots \min_{r_{\tau-2}} \max_{u_{\tau-2}, \theta_{\tau-1}} \sum^{\tau-1}_{n=m}h_{n+1}  \sigma(x_n)\right\} \label{22sums}
\end{equation}

 Denote,
\begin{equation*}
 \widehat{r} = \{r_i\}_{i=0}^{m-1}, \quad \widehat{u} = \{u_i\}_{i=0}^{m-1}, \quad \widehat{\theta} = \{\theta_i\}_{i=0}^{m-1}.
\end{equation*}

We can rewrite \eqref{22sums} as:
%\begin{multline}
% V_\tau(x_0)  = \max_{\theta_0} \min_{r_0} \max_{u_0, \theta_1} \cdots \min_{r_{m-1}} \max_{u_{m-1}}\left \{ \sum^{m-1}_{n=0} h_{n+1} \sigma(x_n)\right.\\
%                             \left. +V_\tau(\alpha(x_0,\widetilde{r},\widetilde{u},m))\right\} \label{222sums}
%\end{multline}
\begin{equation}
 V_\tau(x_0)  = \max_{\theta_0} \min_{r_0} \max_{u_0, \theta_1} \cdots \min_{r_{m-1}} \max_{u_{m-1}}\left \{ \sum^{m-1}_{n=0} h_{n+1} \sigma(x_n)
                              +V_\tau(\alpha(x_0,\widehat{r},\widehat{u},m))\right\} \label{222sums}
\end{equation}

Let $\Xi$ be a bounded subset of $\mathbb{R}^{m}$ and $x_{0}\in \Xi$. Furthermore, let $\rho:\Xi\mapsto\lbrack-1,1]^{m}$ be the function defined in (\ref{badR}) and $\widetilde{X} \subset \mathbb{R}^m$ denote the set of all states that can be reached in exactly $m$ steps for some input sequence $u \in \ell_+(U)$. According to Lemma \ref{finiteX}, the set $\widetilde{X}$ is finite. Let $\widehat{r} = \rho(x_0)$, consequently $\alpha(x_0,\rho(x_0),\widehat{u},m)=x_m \in \widetilde{X}$ for every sequence $\widehat{u}$. Denote $\bar r = \{r_i\}_{i=0}^{m-2}, ~\bar u = \{u_i\}_{i=0}^{m-2}$, and
\begin{equation*}
\bar \nu \left (x_0, \bar r,\bar{u},\widehat{\theta} \right )   \;  {\buildrel\rm def\over =}\; \sum^{m-1}_{n=0} h_{n+1} \sigma(x_n).
\end{equation*}

Let $\bar{r} = \{\rho(x_0)_i\}_{i=0}^{m-2}$ and denote,

\begin{equation*}
\nu \left (x_0, \bar{u},\widehat{\theta} \right )  \;  {\buildrel\rm def\over =}\;  \bar \nu \left (x_0, \bar r,\bar{u},\widehat{\theta} \right )  \bigg |_{\bar{r} = \{\rho(x_0)_i\}_{i=0}^{m-2}} .
\end{equation*}

Thus,
\begin{equation*}
 V_\tau(x_0)   \le \max_{\widehat{u}, \widehat{\theta}} \left \{ \nu \left (x_0,\bar{u},\widehat{\theta} \right ) +V_\tau(x_m)\right\}, 
\end{equation*}
where $x_m \in \widetilde{X}$ for every sequence $\widehat{u}$. Taking supremum over $\tau$ from both sides of the above inequality, we have:
\begin{align*}
 V_\infty(x_0) & \le \sup_\tau \max_{\widehat{u}, \widehat{\theta}} \left \{ \nu \left (x_0,\bar{u},\widehat{\theta} \right ) +V_\tau(x_m)\right\}, \\
                       & = \max_{\widehat{u}, \widehat{\theta}} \left \{ \nu \left (x_0,\bar{u},\widehat{\theta} \right ) +V_\infty(x_m)\right\}.                
\end{align*}

Hence,
\begin{equation}
\sup_{x_0 \in \Xi}  V_\infty(x_0) \le \sup_{x_0 \in \Xi} \max_{\bar{u}, \widehat{\theta}} \nu \left (x_0,\bar{u},\widehat{\theta} \right ) +  \sup_{x_0 \in \Xi} \max_{\widehat{u}, \widehat{\theta}}V_\infty(x_m).   \label{theIneq}
\end{equation}

Since $\Xi$ and $U$ are bounded, the set 
\[
 \left \{ A^nx + \sum_{k=-1}^{n-2}A^{k+1}B(r_{n-k}-u_{n-k}): x \in \Xi, r_k \in [-1,1], u_k \in U \right\}
\]

\noindent is also bounded for every finite $n$. Furthermore, $\sigma$ is BIBO, which immediately implies that the first supremum on the right side of inequality \eqref{theIneq} is bounded. Moreover, $x_m$ can take only a finite number of values for every $x_0 \in \Xi$ and every sequence $\widehat{u}$. Since $V_\infty$ is finite and the supremum over a finite set is finite, the second term on the right side of inequality \eqref{theIneq} is also bounded. Hence $V_\infty$ is BIBO.
\end{proof}

\bigskip

\subsubsection*{\bf Proof of Theorem \ref{boundednessTHM}}
For system \eqref{x+} with exactly one pole $z_1$ on the unit circle and all other poles strictly inside the unit circle, with $e_{1} \in \mathbb{R}^m\backslash \{ 0\}$ such that $Ae_1 = z_1 e_1$, and $Q \in \mathbb{R}^{m \times m}$, $Q = Q' > 0$ such that $Q \ge A'QA$, there exists an invariant cylinder with axis $e_1$ for some $\beta > 0$. Furthermore, the intersection of $\mathcal{C}_{Q,\beta}(e_1)$ with the set $\{|e_1x|<\zeta: x \in \mathbb R^m,\zeta > 0\}$ is bounded whenever $Ce_1 \ne 0$. Define
\begin{equation*}
M_0 =  \sup_{x\in S_0} V_\infty(x).
\end{equation*}

Lemma \ref{supVexists} guarantees finiteness of the supremum. Let $\alpha: \mathbb{R}^m \times \ell_+([-1,1]) \times \ell_+(U) \times \mathbb{Z}_+ \mapsto \mathbb{R}^m$ be a function defined in the proof of Lemma \ref{supVexists}. Let $L$ denote the smallest integer strictly larger than $M_0/\epsilon_0$. Define
%\begin{multline}
%S_L = \left\{x \in \mathcal{C}_{Q,\beta}(e_1) :  \alpha(x,r,u,k) \notin S_0, ~ \forall k \le L,\right.\\
%\left. ~ \forall r \in \ell_+([-1,1]), ~\forall u \in \ell_+(U)  \right\}
%\end{multline}
\begin{equation*}
S_L = \left\{x \in \mathcal{C}_{Q,\beta}(e_1) :  \alpha(x,r,u,k) \notin S_0, ~ \forall k \le L,
 ~ \forall r \in \ell_+([-1,1]), ~\forall u \in \ell_+(U)  \right\}
\end{equation*}
That is, $S_L$ is the set of all states within the invariant cylinder from which $S_0$ cannot be reached in L steps or less. The complement of the set $S_L$ is the region of the cylinder for which there exist sequences $r$ and $u$ such that the state gets to $S_0$ in $L$ steps or less. Since, both $r$ and $u$ are uniformly  bounded, the set $S_L^c$ is bounded.

%%Related to \eqref{Vtau} 

For $j=\{0,1,2, \cdots,\tau-2\}$, define functions $g_j:\mathbb{R}^m \times \ell_+(\{0,1\}) \times \ell_+([-1,1]) \times \ell_+(U) \mapsto \mathbb{R}$,
%\begin{multline}
%g_j(\bar{x},\{\theta_i \}_{i=0}^j,\{r_i\}_{i=0}^{j-1},\{u_i\}_{i=0}^{j-1})=\\
%\min_{r_j} \max_{u_j, \theta_{j+1}} \cdots \min_{r_{\tau-2}} \max_{u_{\tau-2}, \theta_{\tau-1}} \sum^{\tau-1}_{n=0} h_{n+1} \sigma(x_n) \label{g_j}
%\end{multline}
\begin{equation}
g_j(\bar{x},\{\theta_i \}_{i=0}^j,\{r_i\}_{i=0}^{j-1},\{u_i\}_{i=0}^{j-1})=
\min_{r_j} \max_{u_j, \theta_{j+1}} \cdots \min_{r_{\tau-2}} \max_{u_{\tau-2}, \theta_{\tau-1}} \sum^{\tau-1}_{n=0} h_{n+1} \sigma(x_n) \label{g_j}
\end{equation}
\begin{equation}
g_{\tau-1}(\bar{x},\{\theta_i \}_{i=0}^{\tau-1},\{r_i\}_{i=0}^{\tau-2},\{u_i\}_{i=0}^{\tau-2})=  \sum^{\tau-1}_{n=0} h_{n+1} \sigma(x_n)
\end{equation}

For $\epsilon > 0$ and $\bar x \in \mathbb R^m$, let $\widetilde{\tau}(\bar x) \in  \mathbb{Z}_+$,  $\widetilde{\Theta}_{0}(\bar{x}) \in \{0,1\}$, $\widetilde{R}_0(\bar{x})  \in [-1,1] $, $\widetilde{\mathcal{U}}_0(\bar{x}) \in U$, $\widetilde{\Theta}_{1}(\bar{x}) \in \{0,1\}$ be functions such that:
\begin{align}
\widetilde{\tau}(\bar x) & \in \arg^\epsilon \sup_{\tau \in \mathbb{Z}_+} V_\tau(\bar{x}), \label{etau} \\
\widetilde{\Theta}_0(\bar{x}) & =  \theta_0 \in \arg \max_{\theta_0} g_0(\bar{x},\theta_0), 
\end{align}
\begin{equation}
\widetilde{R}_0(\bar{x})  =  r_0 \in  \arg \min_{r_0} \max_{u_0, \theta_{1}}  g_{1}\left(\bar{x},\left(\widetilde{\Theta}_0(\bar{x}),\theta_{1}\right),r_0, u_0 \right), 
\end{equation}
%\begin{multline}
%\left( \widetilde{\mathcal{U}}_0(\bar{x}), \widetilde{\Theta}_{1}(\bar{x}) \right)  = \left(u_0, \theta_{1} \right) \in \\
%\arg \max_{u_0, \theta_{1}}  g_{1}\left(\bar{x},\left(\widetilde{\Theta}_0(\bar{x}),\theta_{1}\right),\widetilde{R}_0(\bar{x}), u_0 \right).
%\end{multline}
\begin{equation}
\left( \widetilde{\mathcal{U}}_0(\bar{x}), \widetilde{\Theta}_{1}(\bar{x}) \right)  = \left(u_0, \theta_{1} \right) \in 
\arg \max_{u_0, \theta_{1}}  g_{1}\left(\bar{x},\left(\widetilde{\Theta}_0(\bar{x}),\theta_{1}\right),\widetilde{R}_0(\bar{x}), u_0 \right).
\end{equation}

For $\bar x \in \mathbb R^m$ and $j=\{0,1,2, \cdots,\widetilde{\tau}(\bar x)-2\}$, let $\widetilde{\Theta} \in \{0,1\}^{j+2}$, $\widetilde{R}  \in [-1,1]^{j+1} $, $\widetilde{\mathcal{U}} \in U^{j+1}$, $\widetilde{R}_j(\bar{x})  \in [-1,1] $, $\widetilde{\mathcal{U}}_j(\bar{x}) \in U$, and $\widetilde{\Theta}_{j+1}(\bar{x}) \in \{0,1\}$ be functions such that:
%\begin{align}
%\widetilde{\Theta} & = \left(\{\widetilde{\Theta}_i(\bar{x})\}_{i=0}^j,\theta_{j+1}\right), \widetilde{R} = \left(\{\widetilde{R}_i(\bar{x})\}_{i=0}^{j-1},r_j\right), \\
%\widetilde{\mathcal{U}} & = \left(\{\widetilde{\mathcal{U}}_i(\bar{x})\}_{i=0}^{j-1},u_j\right).
%\end{align}
\begin{equation}
\widetilde{\Theta} = \left(\{\widetilde{\Theta}_i(\bar{x})\}_{i=0}^j,\theta_{j+1}\right), \quad \widetilde{R} = \left(\{\widetilde{R}_i(\bar{x})\}_{i=0}^{j-1},r_j\right), \quad
\widetilde{\mathcal{U}} = \left(\{\widetilde{\mathcal{U}}_i(\bar{x})\}_{i=0}^{j-1},u_j\right).
\end{equation}

\begin{equation}
\widetilde{R}_j(\bar{x}) =  r_j \in  \arg \min_{r_j} \max_{u_j, \theta_{j+1}}  g_{j+1}\left(\bar{x},\widetilde{\Theta},\widetilde{R}, \widetilde{U} \right), 
\end{equation}
%\begin{multline}
%\left( \widetilde{\mathcal{U}}_j(\bar{x}), \widetilde{\Theta}_{j+1}(\bar{x}) \right) = \left(u_j, \theta_{j+1}\right) \in \\
%\arg \max_{u_j, \theta_{j+1}}   g_{j+1}\left(\bar{x},\widetilde{\Theta},\{\widetilde{R}_i(\bar{x})\}_{i=0}^{j}, \widetilde{\mathcal{U}} \right). \label{U_j}
%\end{multline}
\begin{equation}
\left( \widetilde{\mathcal{U}}_j(\bar{x}), \widetilde{\Theta}_{j+1}(\bar{x}) \right) = \left(u_j, \theta_{j+1}\right) \in 
\arg \max_{u_j, \theta_{j+1}}   g_{j+1}\left(\bar{x},\widetilde{\Theta},\{\widetilde{R}_i(\bar{x})\}_{i=0}^{j}, \widetilde{\mathcal{U}} \right). \label{U_j}
\end{equation}

\bigskip

Assuming that $\bar{x} \in S_L$, there are two cases to consider, either: 
\bigskip
\begin{enumerate}
\item $\alpha(\bar{x},\{\widetilde{R}_i(\bar{x})\}_{i=0}^{J-1},\{\widetilde{\mathcal{U}}_i(\bar{x})\}_{i=0}^{J-1},J) \notin S_0$, for all $J \in \{0, 1, 2, \cdots \widetilde{\tau}-1\}$.

\item There exists an integer $J > L$ such that $\alpha(\bar{x},\{\widetilde{R}_i(\bar{x})\}_{i=0}^{J-1},\{\widetilde{\mathcal{U}}_i(\bar{x})\}_{i=0}^{J-1},J) \in S_0$.
\end{enumerate}

\bigskip

From equations \eqref{g_j}$-$\eqref{U_j}, we have:
\begin{align}
g_0(\bar{x},0) & = 0,\\
g_0(\bar{x},1) & = \min_{\{r_i\}_{i=0}^{\widetilde{\tau}(\bar x)-2}}\sum^{\widetilde{\tau}(\bar x)-1}_{n=0} h_{n+1} \sigma(x_n)\\
                                     & \le \sum^{\widetilde{\tau}(\bar x)-1}_{n=0} h_{n+1} \sigma(x_n)\left|_{\{r_i\}_{i=0}^{\widetilde{\tau}(\bar x)-2}=\{\widetilde{R}_i(\bar{x})\}_{i=0}^{\widetilde{\tau}(\bar x)-2}} \right.
\end{align}

\noindent where  $\{u_i\}_{i=0}^{\widetilde{\tau}(\bar x)-2}=\{\widetilde{\mathcal{U}}_i(\bar{x})\}_{i=0}^{\widetilde{\tau}(\bar x)-2}$ and $\{\theta_i\}_{i=1}^{\widetilde{\tau}(\bar x)-1}=\{\widetilde{\Theta}_i(\bar{x})\}_{i=1}^{\widetilde{\tau}(\bar x)-1}$. Furthermore, by \eqref{Vinf} and \eqref{etau}, we have:
\begin{align}
V_\infty(\bar{x}) & < \epsilon + V_{\widetilde{\tau}(\bar x)}(\bar{x}) = \epsilon + \max_{\theta_0}g_0(\bar{x},\theta_0)\\
                            & = \epsilon + \max\left\{g_0(\bar{x},0), g_0(\bar{x},1)\right\}.
\end{align}

Consider case $(1)$: since $\sigma(x) \le -\epsilon_0$ for all $x\notin S_0$, the sum over $h_{n+1}\sigma(x_n)$ will be negative for all $\widetilde{\tau}(\bar x)$. Thus, 
\begin{equation}
V_\infty(\bar{x}) < \epsilon, \quad \forall \bar{x} \in S_L.
\end{equation}

For case $(2)$, we can write $g_0(\bar{x},1)$ equivalently as:
%\begin{multline}
% g_0(\bar{x},1)  = \min_{r_0} \max_{u_0, \theta_1} \cdots \min_{r_{J-2}} \max_{u_{J-2}, \theta_{J-1}}\left \{ \sum^{J-1}_{n=0} h_{n+1} \sigma(x_n)\right.\\
%                        \hspace{.1in}               \left. +\min_{r_{J-1}} \max_{u_{J-1}, \theta_J} \cdots \min_{r_{\widetilde{\tau}(\bar x)-2}} \max_{u_{\widetilde{\tau}(\bar x)-2}, \theta_{\widetilde{\tau}(\bar x)-1}} \sum^{\widetilde{\tau}(\bar x)-1}_{n=J}h_{n+1}  \sigma(x_n)\right\} \label{2sums}
%\end{multline}
\begin{equation}
 g_0(\bar{x},1)  = \min_{r_0} \max_{u_0, \theta_1} \cdots \min_{r_{J-2}} \max_{u_{J-2}, \theta_{J-1}}\left \{ \sum^{J-1}_{n=0} h_{n+1} \sigma(x_n)
                              +\min_{r_{J-1}} \max_{u_{J-1}, \theta_J} \cdots \min_{r_{\widetilde{\tau}(\bar x)-2}} \max_{u_{\widetilde{\tau}(\bar x)-2}, \theta_{\widetilde{\tau}(\bar x)-1}} \sum^{\widetilde{\tau}(\bar x)-1}_{n=J}h_{n+1}  \sigma(x_n)\right\} \label{2sums}
\end{equation}
Since the first summation term in \eqref{2sums} is bounded above by $-J\epsilon_0$ for $\{\widetilde{\Theta}_i(\bar{x})\}_{i=0}^{J-1}$ and all sequences $\{r_i\}_{i=0}^{J-2}$ and $\{u_i\}_{i=0}^{J-2}$, and the second summation term is equal to $g_0(\alpha(\bar{x},\{r_i\}_{i=0}^{J-1},\{u_i\}_{i=0}^{J-1},J), \theta_{J})$,  we have:
%\begin{multline}
% g_0(\bar{x},1) \le \min_{r_0} \max_{u_0, \theta_1} \cdots \min_{r_{J-2}} \max_{u_{J-2}, \theta_{J-1}}\left \{ -J\epsilon_0 \right. \\
%		\left.+g_0(\alpha(\bar{x},\{r_i\}_{i=0}^{J-1},\{u_i\}_{i=0}^{J-1},J), \theta_{J})\right\}. \label{sum-J}
%\end{multline}
\begin{equation}
 g_0(\bar{x},1) \le \min_{r_0} \max_{u_0, \theta_1} \cdots \min_{r_{J-2}} \max_{u_{J-2}, \theta_{J-1}}\left \{ -J\epsilon_0 
		+g_0(\alpha(\bar{x},\{r_i\}_{i=0}^{J-1},\{u_i\}_{i=0}^{J-1},J), \theta_{J})\right\}. \label{sum-J}
\end{equation}
Since 
\begin{equation}
g_0(\alpha(\bar{x},\{\widetilde{R}_i(\bar{x})\}_{i=0}^{J-1},\{\widetilde{\mathcal{U}}_i(\bar{x})\}_{i=0}^{J-1},J),\widetilde{\Theta}_{J}(\bar{x})) \le M_0,
\end{equation}
we have:
\begin{equation}
 g_0(\bar{x},1)  \le M_0-J\epsilon_0.
\end{equation}

Furthermore, $J> M_0/\epsilon_0$, thus,
\begin{equation}
V_\infty(\bar{x}) < \epsilon + \max \left \{0, ~M_0-J \epsilon_0\right \} = \epsilon.
\end{equation}

Since, $V_\infty(\bar{x}) < \epsilon$ for every $\epsilon > 0$, 
\begin{equation}
V_\infty(\bar{x}) = 0, \quad \forall \bar{x} \in S_L.
\end{equation}

Finally, the complement of the set $S_L$ is bounded; therefore, \eqref{Vne0Bounded} is bounded.

%%%%%%%%%%%%%%%%%%%%%%%%%%%%%%%%%%%%%%%%%%%%%%%%%%%%%%%%%%%%%%%%%

\nocite{PhDBib:Sasha_FSA} \nocite{PhDBib:Mitra0} \nocite{PhDBib:ROC1996} \nocite{PhDBib:Basar95}
\bibliographystyle{IEEEtran}
\bibliography{acompat,My_PhD_Bibliography}

\end{document}